\documentclass[11pt]{amsart}
\usepackage{amssymb, amsmath, amsthm}
\usepackage[latin1]{inputenc}
\usepackage{graphicx}
\graphicspath{{./pics/}}
\usepackage[final]{hyperref}
\usepackage{color}
\usepackage{ulem}
\usepackage{epstopdf}

\usepackage[a4paper, centering]{geometry}
\usepackage{color}
\usepackage{graphicx}
\usepackage{scalerel,stackengine}

\usepackage{mathtools}

\newtheorem{theorem}{Theorem}
\newtheorem{proposition}[theorem]{Proposition}
\newtheorem{lemma}[theorem]{Lemma}

\theoremstyle{definition}
\newtheorem{definition}[theorem]{Definition}
\theoremstyle{remark}
\newtheorem{remark}[theorem]{Remark}

\def\R{\mathbb{R}}

\def\N{\mathbb{N}}

\def\Z{\mathbb{Z}}
\def\A{\mathcal{A}}
\def\c {\,\,\, ^ {4} \!\!\! \! \sqrt {12} }
\def\t{\Delta}
\def \wt{\widehat \Delta} 	
\def \wwt{\widehat {\widehat  \Delta}}

\begin{document}

\title[]{PROOF OF  THE HONEYCOMB ASYMPTOTICS \\ FOR OPTIMAL CHEEGER CLUSTERS}
\author[]{Dorin Bucur, Ilaria Fragal\`a}
\thanks{}

\address[Dorin Bucur]{Institut Universitaire de France et
Laboratoire de Math\'ematiques CNRS UMR 5127 \\
Universit\'e  Savoie Mont Blanc,  Campus Scientifique \\
73376 Le-Bourget-Du-Lac (France)
}
\email{dorin.bucur@univ-savoie.fr}

\address[Ilaria Fragal\`a]{
Dipartimento di Matematica \\ Politecnico  di Milano \\
Piazza Leonardo da Vinci, 32 \\
20133 Milano (Italy)
}
\email{ilaria.fragala@polimi.it}


\begin{abstract}
{ We prove that, in the limit as $k \to+ \infty$,  the hexagonal honeycomb solves the optimal partition problem in which the criterion is minimizing the largest among the Cheeger constants of $k$ mutually disjoint cells in a planar domain.
As a by-product, the same result holds true when the Cheeger constant is replaced by the first Robin eigenvalue of the Laplacian. 
 }
   \end{abstract}
\maketitle

\section{Introduction and statement of the results}

Consider the following optimal partition problem
\begin{equation}\label{f:problemk} M_k(\Omega)= \inf  \Big \{  \max_{j = 1, \dots , k}  h  (\Omega _j )  \ :\  \{\Omega_j\} \in \A _k (\Omega) \Big \}\, , 
\end{equation}
where $\Omega$ is an open bounded subset of $\R^2$ with a Lipschitz boundary, $h (\cdot )$ is the Cheeger constant, and ${\A} _k (\Omega)$ is the class of  
$k$-clusters of $\Omega$, meant as families
of $k$ Borel sets with finite perimeter which are contained into $\Omega$ and have Lebesgue negligible mutual intersections.

Let us recall that the  Cheeger constant of $\Omega$ is defined by
\begin{equation}\label{f:defh}
h (\Omega):=\inf \left \{ \frac{{\rm Per} (E, \R ^2)}{|E|}\ :\ E \hbox{ measurable}\, , \ {E\subseteq   \Omega} \right \}\,,
\end{equation}
where ${\rm Per}(  E, \R ^2)$ denotes the perimeter of $E$ in the sense of De Giorgi.
We refer to the review papers \cite{Leo, Pa} and the numerous references therein for an account of the broad literature about the Cheeger constant. 

Optimal partitions for the Cheeger constant have been firstly studied by Caroccia in \cite{Car17}, where he gives some existence and regularity results for the similar problem
\begin{equation}\label{f:pb22}
\displaystyle m _{k} (\Omega) = \inf  \Big \{ \sum_{j = 1, \dots, k}  h (\Omega _j) \ :\  \{\Omega_j\} \in \A _k (\Omega) \Big \}\, .  
\end{equation}
The main motivation he brings to study problem \eqref{f:pb22} is finding  
some bound for the same problem for the first Dirichlet eigenvalue of the Laplacian, $\lambda _ 1 (\Omega)$.    (Recall indeed that 
$\lambda _ 1 (\Omega)$ is bounded from below by $(h (\Omega) /2) ^ 2 $, as proved by Cheeger himself  in \cite{Ch}.)  
Actually, for problem  \eqref{f:pb22} with $\lambda _1$ in place of the Cheeger constant,  
  a long-standing conjecture by Caffarelli and Lin predicts that, in the limit as $k \to + \infty$, an optimal configuration is given by a packing of regular hexagons \cite{CaffLin}\footnote{The question is commonly  known in the literature  (see {\it e.g.}\ \cite{Car17}) as the Caffarelli-Lin conjecture: in fact  a precise mathematical formulation was given in \cite{CaffLin}, along with the first asymptotic estimates.  Nevertheless, the history of the problem seems to be longer. The first predictive formulation of the conjecture appears in a list of open problems proposed by K. Burdzy in a conference in Matrei in 2005  \href{https://people.kth.se/~laptev/ESF/05/Matrei/problems.html}{\nolinkurl{https://people.kth.se/~laptev/ESF/05/Matrei/problems.html}},  and it is motivated by some  numerical computations originating in older papers \cite{bhim96, cbh05} modeling particle systems. It is also worth to notice that a related version of this conjecture, involving the minimization of the maximum among the first eigenvalues of the cells,  appears to be mathematically formulated  in a paper by B. Helffer,  T. Hoffmann-Ostenhof and S. Terracini, who learned the question from M. van den Berg (see \cite{HeHoTe}).}.

More recently, problems of the kind \eqref{f:problemk} and \eqref{f:pb22} have been studied in \cite{bfvv17}, where it is shown that, 
under the a priori requirement that all the cells of the partitions are convex, 
 the honeycomb conjecture holds true under the form 
\begin{equation}\label{f:asymptotics}
\lim_{k \to + \infty}\frac{|\Omega| ^ {1 /2}} {k ^ {1/2}}  M _k (\Omega)  =
h ( H) \,, \qquad 
\lim_{k \to + \infty}\frac{|\Omega| ^ {1 /2}} {k ^ {3/2}}  m _k (\Omega)  =
h ( H) \,,
\end{equation}
where $H$ denotes the unit area regular hexagon. 

Clearly, the convexity assumption  made on the cells in \cite{bfvv17} is quite stringent.  Nevertheless, as a first approach, it seemed reasonable to attack the problem under this restriction, 
since, also in the case of  perimeter minimizing partitions, the case of convex polygonal cells was much simpler and indeed it was settled a long time before the celebrated result by Hales \cite{Hales}
(see Fejes T\'oth \cite{FT64}). 

Goal of this paper is to prove the honeycomb conjecture for the Cheeger constant in full generality, {\it i.e.} with no convexity constraint on the cells.

We focus our attention on problem \eqref{f:problemk}. 
Our strategy consists in considering first the case when $\Omega$ has a special geometry, that for the sake of simplicity we assume to be that of an equilateral triangle $\mathcal T$ (but other shapes, for instance a rectangle, could do the same job), and obtaining  an inequality for $M _k (\mathcal T)$, with $k$ fixed.   The choice of treating first the case of a simple geometry comes along with our variational approach of the inequality: we work with an optimal partition and take significant advantage from optimality. For that reason we need to have a complete and simple description of this one. The conjecture will follow in full generality, once this special geometric case is proved.  In order to deal with $M _k (\mathcal T)$,  
we introduce the auxiliary problems 
\begin{equation}\label{f:problemp} M _{k,p} (\mathcal T) = \inf  \Big \{  \Big [\sum_{j = 1}^k  h^p  (\Omega _j )\Big ] ^ {1/p}  \ :\  \{\Omega_j\} \in \A _k ({\mathcal T}) \Big \}\,, \qquad p  \geq 1\,,
\end{equation}
 and we set
$$\widetilde M _{k,p} (\mathcal T) := \max _{j=1, \dots, k} h (\Omega ^ p _j) \, ,$$
being $\{\Omega ^ p _1, \dots, \Omega ^ p _k \}$ an optimal cluster for problem \eqref{f:problemp}.
Note that there is an abuse of notation, since $\widetilde M _{k,p} (\mathcal T)$ depends on the choice of $\{\Omega ^ p _1, \dots, \Omega ^ p _k \}$, but we keep this simple notation as the dependence on the optimal cluster is not important for our purposes.

It is easy to see that $\widetilde M _{k, p} (\mathcal T)$ converges to $M _k (\mathcal T)$ in the limit as $p \to + \infty$ (see Section \ref{sec:proof}).

Then we prove that both $\widetilde M _{k, p} (\mathcal T)$ and  $M _k (\mathcal T)$ satisfy the following hexagonal lower bound, being $k$ fixed:

\begin{theorem}\label{t:truehoney}
 Let $\mathcal T$ be an equilateral triangle.   For every $p \geq 1$, there holds
\begin{equation}\label{aux}
\frac{{|\mathcal T}|^\frac 12}{k^\frac 12}\widetilde M _{k, p} ({\mathcal T})\ge h(H).
\end{equation}
Consequently, we have
\begin{equation}\label{tesihoney}
\frac{{|\mathcal T|}^\frac 12}{k^\frac 12}M_k({\mathcal T})\ge h(H).
\end{equation}
\end{theorem}

Theorem \ref{t:truehoney} is the keystone of our approach. Hereafter is an attempt of enlightening the main ideas upon which our proof
is based:
\begin{itemize}
\item{} Optimal clusters satisfy an existence, regularity, and structure result, which is essentially a variant of the one valid in the case $p=1$ treated by Caroccia (see Proposition \ref{p:optimal}). 
\item{} As a consequence of the structure result, each cell of an optimal cluster  is Cheeger of itself and enjoys the following key property. If we call ``inner Cheeger boundary'' of a cell the inner parallel set at distance to the boundary  equal to the inverse of the Cheeger constant, then the  oriented area enclosed by such  inner Cheeger boundary turns out to be related to the Cheeger constant of the cell itself by a very simple equation (see Proposition \ref{p:representation}, eq.\eqref{f:area}). It can be read as the transposition of  a well-known relation between the Cheeger set of convex bodies and their inner parallel sets. In turn, this leads to a crucial representation formula 
for the Cheeger constant of an optimal cell in terms of its area and of the length of its inner Cheeger boundary
(see Proposition \ref{p:representation}, eq.\eqref{f:rep}). Such representation formula can be regarded as the initial seed of our proof. 
  \item{} Starting from the representation formula, the optimality of the hexagonal honeycomb comes out by combining a {\it lower bound for the total length of the inner Cheeger boundaries}, with an {\it upper bound for the total area of the cells}. Both are quite delicate. In particular, the former is obtained by applying 
   Hales' hexagonal isoperimetric inequality \cite[Theorem 4]{Hales}, going through the analysis of the collective behaviour of those inner Cheeger boundaries. The latter requires a careful estimate of the area of the empty chamber, which is carried over through some topological and geometrical arguments (see Proposition \ref{p:void}). 
\end{itemize}

 Next,  as a consequence of Theorem \ref{t:truehoney}, we are able to consider the case when the  equilateral triangle is replaced by $k$-triangle, that is a region of triangular shape formed by $k$ hexagons. More precisely,  for  $\displaystyle k=l(l+1)/2$, 
by {\it $k$-triangle}, we mean a connected set  which is obtained as the union of $k$ hexagons lying in a tiling of $\R ^2$ made by a family of copies of a regular hexagon and having the ``rough" shape of an equilateral triangle with $l$ cells on each side (precisely, all the centers of those hexagons lie on the boundary and inside an equilateral triangle).  

We obtain that, for any fixed $k$, the energy of a $k$-triangle (denoted by $\mathcal T _k$), suitably scaled, is precisely that of the regular hexagon:

\begin{theorem}\label{t:truehoney_bis}
Let $\mathcal T_k$ be a $k$-triangle.  
 There holds
\begin{equation}\label{tesihoney_bis}
\frac{{|\mathcal T_k|}^\frac 12}{k^\frac 12}M_k({\mathcal T_k})= h(H).
\end{equation}

\end{theorem}

Finally, relying on Theorem \ref{t:truehoney_bis} and using a  blow-up argument, we obtain  that 
the honeycomb conjecture for the Cheeger constant holds true for every Lipschitz domain $\Omega$ in the following asymptotic form
(which is exactly the same as in \cite{bfvv17}, without the convexity assumption on the cells):

\begin{theorem}\label{t:truehoney_ter}
For every open bounded Lipschitz domain $\Omega$, and every $p \geq 1$, there holds
\begin{equation}\label{tesihoney_ter}
\lim_{k \to + \infty}\frac{|\Omega| ^ {1 /2}} {k ^ {1/2}}  M _{k}  (\Omega)  =
h ( H)
\end{equation}
\end{theorem}

\begin{remark}[\bf Asymptotic behaviour for partitions of the Robin-Laplacian eigenvalues]

It is worth noticing that, as a consequence of the above result and Corollary 3 (i) together with Remark 15 in \cite{bf17R}, 
the same result as Theorem \ref{t:truehoney_ter} holds true if in the definition of $M _k (\Omega)$ the Cheeger constant $h(\Omega_j)$ is replaced by the first 
eigenvalue of the Laplacian under Robin boundary  conditions, $\lambda_1(\Omega_j, \beta)$.  Precisely, given $\beta >0$ (fixed), $\lambda_1(\Omega_j, \beta)$ is the lowest positive number for which the equation
$$
\begin{cases}
- \Delta u = \lambda_1(\Omega_j, \beta) u & \text{ in } \Omega_j
\\
\frac{\partial u }{\partial \nu} + \beta u = 0 & \text{ on } \partial \Omega_j\,.
\end{cases}
$$
has a non trivial solution.
\end{remark}

\begin{remark}[\bf Weyl asymptotic for the $k$-th Cheeger constant]
The quantity $M_k(\Omega)$ is also called the $k$-th Cheeger constant of $\Omega$  (see the recent paper \cite{PaBo} and references therein), 
and an equivalent of this notion is intensively studied on graphs, for clustering purposes (see {\it e.g.}\ \cite{Trev}). 

Loosely speaking, for the 1-Laplacian operator, the $k$-th Cheeger constant can be interpreted as a counterpart of  the $k$-th eigenvalue. 
In this perspective, Theorem \ref{t:truehoney_ter} can also be interpreted as an asymptotic formula of Weyl type \cite{Ivrii} for the $k$-th Cheeger constant, since it can be rephrased as 
$$ M_k(\Omega) = \frac{k^\frac 12}{|\Omega|^\frac 12}(\sqrt{\pi}+ \sqrt[4]{12}) + o(k^\frac 12).$$ 

\end{remark}

\medskip

The plan of the paper is the following. In Section \ref{sec:structure}, we establish all the preparatory results which concern the properties of optimal clusters; the results of this section relay on the work of Caroccia \cite{Car17}. Next we give the intermediate results of topological nature in Section \ref{sec:top}, and the key representation result involving the inner Cheeger boundary in Section \ref{sec:bdry}. 
 The proofs of Theorems \ref{t:truehoney},   \ref{t:truehoney_bis},  and  \ref{t:truehoney_ter},  
are  then given in Section \ref{sec:proof}.  Finally in Section \ref{sec:app},  we collect some auxiliary geometrical lemmas needed for the estimate of the area of the empty chamber.

\section {About optimal clusters}\label{sec:structure}

This section is devoted to the study of optimal clusters for problem \eqref{f:problemp}, in case $\Omega$ is an equilateral triangle $\mathcal T$: 
in Section \ref{subsec:structure} we give a structure result along the same line of the one proved by Caroccia for $p=1$; in Section \ref{sec:important} we fix some important consequences of the structure result; in Section \ref{sec:graph} we associate with an optimal cluster a planar graph,
which will be used as fundamental tool to establish the topological results stated in the next section.

\subsection{A structure result for optimal clusters}\label{subsec:structure}

\begin{definition}\label{defA}  We denote by $\mathcal A$ the family of Jordan domains $\Omega$ of class $C ^ {1}$ contained into $\mathcal T$ such that $\Omega$ is Cheeger set of itself, and the  (positively oriented) boundary $\partial \Omega$
 is the union of an even number of  nontrivial arcs alternating
 a {\it free arc} and a {\it junction arc}. A free arc is an arc of circle of algebraic curvature $h (\Omega)$ with at least one  endpoint in the interior of $\mathcal T$. A junction arc may be either an {\it inner junction arc} or a {\it border junction arc}. 
 An inner junction arc is an arc of circle of algebraic curvature 
 strictly less than $h (\Omega)$
 (possibly $0$),  with both the endpoints  in the interior of $\mathcal T$. An outer junction arc is a curve,  with both the endpoints on $\partial \mathcal T$,  which is union of segments lying in $\partial \mathcal T$ and arcs of circle of curvature $h (\Omega)$. 
\end{definition}

\begin{remark}\label{r:palla}
We point out that $\mathcal A$  does not contain any ball, firstly  because the number of circular arcs must be even, and also because, if $B$ is a ball of radius $R$, it holds  $h (B) = \frac{2}{R}$, so that the curvature is not equal to $h (B)$. For a similar reason, $\mathcal A$ does not contain any stadium-domain (that is, the convex envelope of two balls).   As a further example it is easy to check that, among all convex domains obtained from a square by ``rounding off'' the corners with four circular arcs, only the Cheeger set of the square lies in the class $\A$. 
\end{remark}
\begin{proposition}[properties of an optimal cluster] \label{p:optimal}  
For every fixed $p \geq 1$:  

\smallskip
\begin{itemize}
\item[(i)]  problem \eqref{f:problemp} admits a solution in which each cell is Cheeger of itself, hereafter denoted by $\{ \Omega _1, \dots , \Omega _k \}$;
 
 \smallskip
 \item[(ii)] each cell $\Omega _j$ is a simply connected set of class $C^1$;

\smallskip
\item[(iii)] each cell $\Omega_j$ belongs to the family
$\mathcal A$ introduced in Definition  \ref{defA};  moreover:

\begin{itemize}

\item[-- ]   any inner junction arc for $ \Omega_j$ is also an inner junction arc for another set $\Omega_l$, and its curvature, seen from $\Omega _j$,  is given by 
\begin{equation}\label{curv} K_{j,l} = \displaystyle
\frac
{ \frac{h ^ p ( \Omega_j)} {|\Omega_j|} 
-\frac{h ^ p (\Omega_l)} {|\Omega_l|}
}
{ \frac{h ^ {p-1} ( \Omega_j)} {|\Omega_j|} 
 +  \frac{h ^ {p-1} ( \Omega_l)} {|\Omega_l|}
}\,;
\end{equation} 

\item[--]  any free arc for $\Omega_j$ can intersect $\cup _{l \neq j} \partial \Omega_l$ on at most a finite number of points; moreover, the opening angle of any portion of a free arc which does not contain intersection points with $\cup _{l\neq j} \partial \Omega _l$ is strictly less than $\pi$.
\end{itemize}
\end{itemize} 
\end{proposition}
\begin{proof}
For $p=1$ the existence, regularity and structure of optimal clusters of problem  \eqref{f:problemp} have been discussed in \cite{Car17}. For $p>1$, the arguments are precisely the same, without any significant difference. For the convenience of the reader, we highlight the main steps, and refer to \cite{Car17} for details. 

\medskip
(i) We replace the original problem \eqref{f:problemp} by the following one 
\begin{equation}\label{f:problemp-d}\inf  \Big \{  \Big [\sum_{j = 1}^k  \Big (\frac {\mathcal H ^ 1(\partial ^*\Omega _j)}{|\Omega_j|} \Big )^ p\Big ] ^ {1/p}  \ :\  \{\Omega_j\} \in \A _k ({\mathcal T}) \Big \}\, .
\end{equation}

We trivially get  an upper bound  for the value 
of the above infimum by referring to some configuration (e.g. $k$ disjoint balls). 
As a consequence, there exists a constant $M>0$ such that, for any minimizing sequence $(\Omega^n _1, \dots, \Omega^n_k)$, 
it holds 
$$\sum_{j = 1}^k  \Big (\frac {\mathcal H ^ 1(\partial ^*\Omega^n _j)}{|\Omega^n _j|} \Big )^ p\le M^p.$$
Combined with the isoperimetric inequality, this implies that the measures $|\Omega^n_j|$ remain bounded from below. 
As well, we get the upper bound
$$\mathcal H ^ 1(\partial ^*\Omega^n _j)\le M |\mathcal T |  .$$ 
Consequently, the existence of optimal clusters for problem \eqref{f:problemp-d} follows by standard compactness/lower semicontinuity arguments in $BV$. Each set of an optimal configuration is self Cheeger, otherwise this would contradict optimality. Moreover, every solution to problem \eqref{f:problemp-d}  is also solution to the original problem \eqref{f:problemp}. 
Let us denote such a solution by $(\Omega_1, \dots, \Omega_k)$.

\medskip
(ii) All sets $\Omega_j$ of the optimal cluster obtained in statement (i) are (locally, inside ${\mathcal T}$) quasi-minimizers for the perimeter, and hence they are equivalent to open sets with boundary having $C^{1,\alpha}$ regularity (inside ${\mathcal T}$), with any $\alpha \in (0, \frac 12)$.

  For $p=1$ the quasi minimality argument is given in \cite[Theorem 3.6]{Car17}, but the proof does not depend on $p$. Roughly speaking, together with the minimality in  \eqref{f:problemp-d}, the key points are that each set $\Omega_j$ is self Cheeger and has an algebraic curvature (in a distributional sense) not larger than a constant (in our case $M$). So, we know that each set $\Omega_j$ is equivalent to an open set with smooth boundary. A priori, the set $\Omega_j$ may not be connected. In case that $\Omega_j$ is not connected, two connected components have necessarily to lie at positive distance, and we can choose one of them and replace $\Omega_j$ with this component. The energy  in  \eqref{f:problemp-d}  does not change. So we know that all sets $\Omega_j$ are connected. As a consequence, the vertices of ${\mathcal T}$ do not belong to any of the boundaries $\partial \Omega_j$, since cutting out by a line a piece of $ \Omega_j$ near the corner, would strictly decrease its Cheeger constant.

Moreover, each set is simply connected. Indeed, if a set $\Omega_j$ is not simply connected, we analyze one hole (which is smooth) and translate it inside $\Omega_j$ up to a new contact point with $\partial \Omega_j$. This new set is also optimal, contradicting the regularity.

\medskip
(iii) We analyze now the structure of the boundary. Following the same arguments as for $p=1$ in \cite[Proposition 5.4 and Proposition 5.5]{Car17}, there are no triple points (meaning that a point of $\overline {\mathcal T}$ may belong to at most two boundaries $\partial \Omega_j$, $\partial \Omega_l$), and the boundary of $\Omega_j$ is a finite union of arcs of circle. Moreover, looking at a piece of arc of circle which is common to $\partial \Omega_j$ and $\partial \Omega_l$, one can write optimality conditions, which lead precisely to the expression of the algebraic curvature (seen from $\Omega_j$) given by \eqref{curv}. 
We see from \eqref{curv} that $K_{j,l}$ is strictly less than $h(\Omega_j)$. If we look now at a piece of arc of circle from $\partial \Omega_j$ lying in a neighborhood of a point which has a positive distance from $\cup _{l\neq j} \partial \Omega _l$, we get from optimality that the curvature has to be equal to $h(\Omega_j)$. As a consequence of the $C^1$-regularity, two such pieces of arc from $\partial \Omega _j$ meeting at a point which belongs to  $\cup _{l\neq j} \partial \Omega _l$ have to be part of  a unique arc of curvature $h(\Omega_j)$. In this way, we identify clearly the boundary of $\Omega_j$ as an ordered union of free arcs of circle of algebraic curvature equal to $h(\Omega_j)$ alternating with junction arcs  which may be inner or border ones.  \end{proof}

\begin{remark}\label{r:ambrosio}
We point out that part of the information on the properties of an optimal cluster given in Proposition \ref{p:optimal} could be obtained in a direct way by applying 
to each cell a structure result by Ambrosio, Caselles, Masnou and  Morel
for measurable sets with finite perimeter in two dimensions, which are indecomposable in the sense of geometric measure theory (see \cite{ACMM01}).  
\end{remark} 

\subsection{Consequences of the structure result.}\label{sec:important}

As an outcome of Proposition \ref{p:optimal}, the structure of an optimal cluster for problem \eqref{f:problemp} 
is quite rigid. For later use, it is important to fix in particular the following facts.

\medskip

-- {\it Connected components of the empty chamber}. By empty chamber, we mean the set $\Omega _0:= \mathcal T \setminus \cup _{j=1} ^m \overline \Omega _j$. 
Every connected component $c (\Omega _0)$ of the empty chamber is a Jordan domain.  If  $c (\Omega _0)$   has a positive distance from $ \partial \mathcal T$, its boundary is  a union of free arcs. If  $c (\Omega _0)$ touches $\partial \mathcal T$, two possibilities may occur: either $\partial c (\Omega _0)$   is  union of some free arcs and some segments on $\partial \mathcal T$, or  $\partial c (\Omega _0)$  is  union of two segments lying on consecutive sides of $\partial \mathcal T$ and a piece of a border junction arc (and this may occur only around the corners).   

We point out in particular that, if we endow  $\partial c (\Omega _0)$  with a positive orientation, all the  arcs of circle  have negative curvatures $-h(\Omega_j)$, being $\Omega_j$ the neighbouring cells. Thus, as a consequence of the sign of the curvatures, 
 $\partial c (\Omega _0)$ contains at least three arcs (meant as arcs of circle or segments).

\medskip
-- {\it Cells sharing several  inner  junction arcs}.  
Two cells $\Omega_j, \Omega_l$ may share several  inner  junction arcs. In this case, two consecutive  inner junction arcs need to enclose another cell. More precisely assume that, following the orientation of $\partial \Omega _j$, we find  two consecutive  inner  junction arcs
$\gamma _1$ and $\gamma_2$, and let us denote by  $P$ the final point  of $\gamma _1$ and by $Q$ the initial  point of  $\gamma _2$. Consider the curve $\gamma$ starting on $P$, following $\partial \Omega_1$ up to $Q$ and then following $\partial \Omega_2$ up to $P$ (still in the positive sense). Then, $\gamma$ has necessarily to enclose another cell, different from $\Omega _j, \Omega _l$. Indeed, $\gamma$ does not contain any other  inner  junction arc between $\Omega_1$ and $\Omega_2$, because we have chosen two consecutive  inner  junction arcs. Thus, the only possibility for $\gamma$ not to enclose another cell would be that on $\partial \Omega_1$ the curve from $P$ to $Q$  is a free arc, and on $\partial \Omega_2$ the curve from $Q$ to $P$ is also a free arc. This is not possible, since the curvature of both free arcs, seen from $\Omega _j$ and $\Omega _l$,  are positive. 

\medskip
--  {\it Cells sharing several border junction arcs with $\partial {\mathcal T}$.} 
A cell $\Omega _j$ may share several border junction arcs with $\partial {\mathcal T}$.
However, in the alternation of free and junction arcs, two border junction arcs cannot be consecutive. 
Indeed, between two border junction arcs, there is a free arc having one endpoint in the interior of $\mathcal T$, so that an inner junction arc starts at such endpoint. 
Notice also that a border junction arc  may contain different segments lying in  $\partial \mathcal T$; if this is the case, due 
to the sign of the curvature of the free arcs, these segments cannot lie on the same side of $\mathcal T$, but belong necessarily to distinct sides of $\mathcal T$;  consequently, they can be either $2$, or  at most $3$.

\subsection{Construction of the canonical graph associated with an optimal cluster.}\label{sec:graph}

Thanks to the properties of an optimal cluster for problem \eqref{f:problemp} described so far, we are ready to associate with it a planar graph. 

\begin{definition}\label{def:graph} We call {\it  canonical graph} associated with an optimal cluster for problem \eqref{f:problemp} the planar graph having the following vertices and edges:
 
-- {\it Vertices}: To each cell $\Omega_j$, $j=1, \dots, k$ we associate a vertex $X_j$. 
Also to the set $\R^2 \setminus \overline {\mathcal T}$ we associate a vertex, denoted $X_0$. 
We have thus $k+1$ vertices. To draw a representation of the graph in the plane, the vertices can be chosen as arbitrary points in the interior of $\Omega _j$ and 
$\R^2 \setminus \overline {\mathcal T}$ respectively. 

 -- {\it Edges}.   
We distinguish the families 
 $\mathcal E_{in}$  and $\mathcal E _{out}$ of  {\it inner} and {\it outer} edges, namely 
edges of the graph which join  two distinct  vertices $X_j$, $X_l$  ($j, l \in \{1, \dots, k\}$), or 
a vertex $X_j$ with  $X_0$, respectively. The family $\mathcal E _{in}$ is constructed as follows: 
 to every couple ($\Omega_j, \Omega_l$) which share an  inner  junction arc, we associate an edge by joining $X_j$ to $X_l$
through such arc. 
 The family  $\mathcal E _{out }$ is constructed as follows:  
 to every cell $\Omega_j$ having a border junction arc on $\partial {\mathcal T}$, we associate an edge by joining $X_j$ to $X_{0}$ 
 through such arc. 
\end{definition}
 
\begin{remark}\label{r:faces}  Each face of 
the canonical graph associated with an optimal cluster for problem \eqref{f:problemp} 
has at least $3$ edges. 
To prove this claim, it's enough to observe that a face of the graph can be delimited neither by just two inner edges nor by just two outer edges. 

Indeed, two cells may share several inner junction arcs, so that two vertices $X_j$, $X_l$ may be connected by multiple inner edges;
however, we know from Section \ref{sec:important}
that two consecutive inner junction arcs need to enclose another cell, and therefore
 no face of the graph can be delimited by just two inner edges. 

Likewise, a cell may share several border junction arcs with $\partial \mathcal T$, but
we know from Section \ref{sec:important} that they cannot be consecutive, and therefore
 no face of the graph can be delimited by just two outer edges. 
\end{remark}

\section{Intermediate topological results}\label{sec:top}

 In this section we give two results needed for the proof of Theorem \ref{t:truehoney}, both obtained via the analysis of the canonical graph associated with an optimal cluster: in Proposition \ref{p:nb} we give an upper bound for the average of the number of junction arcs, and in Proposition \ref{p:void} we provide an estimate from below for the area of the empty chamber.

\begin{proposition}[average of number of junction arcs]\label{p:nb}
For a fixed $p\geq 1$, let $\{\Omega _1, \dots, \Omega _k\}$ be an optimal cluster for problem \eqref{f:problemp} 
which satisfies the properties stated in Proposition \ref{p:optimal}.  Let $ 2\Lambda _j  $, $j = 1, \dots, k$ be the number of oriented arcs which compose
$\partial \Omega _j$, so that $\Lambda_j$ is the number of junction arcs in $\partial \Omega _j$.   Let  $E _{out}$ be the cardinality of the family  $\mathcal E _{out}$ of outer edges of the canonical graph associated with the cluster $\{\Omega _1, \dots, \Omega _k\}$ according to Definition \ref{def:graph}. 
Then the following inequality holds: 
$$\sum_{j=1}^k  \Lambda_j  + E_{out} +6  \le 6k  .$$
\end{proposition}

\begin{proof}

We denote by $V= k+1$ the number of vertices, $E$  the number of edges, and $F$ the number of sides
in the canonical graph associated with the optimal cluster $\{\Omega _1, \dots, \Omega _k\}$. 

Since every edge borders $2$ faces, and each face has at least $3$ edges (by Remark \ref{r:faces}), there holds 
$$2E\ge 3F\,.$$ Then, using the Euler formula 
$$V-E+F=2,$$
we obtain
$$3k-3 \ge E.$$
 Setting $E_{in}$  and $E _{out}$  the cardinalities of the families $\mathcal E_{in}$  and $\mathcal E _{out}$ of inner and outer edges according to Definition \ref{def:graph},  we get
$3k-3 \ge E_{in} + E_{out}$,
or equivalently
\begin{equation}\label{precount}
6k -6 \ge 2E_{in} + 2E_{out}.
\end{equation}
 Then the conclusion is obtained by noticing that 
\begin{equation}\label{count}
  2E_{in} + E_{out}  = \sum_{j=1}^k  \Lambda_j .   
\end{equation}
Indeed, let's count the total number of junction arcs:
any  inner junction arc is counted twice, and corresponds to 
an edge in $\mathcal E _{in}$; any border junction arc is counted once, and corresponds to an edge in $\mathcal E _{out}$.  
\end{proof}

\begin{definition} \label{def:regin}
We set:
\begin{itemize}
\item[--] $\t_r$ the curvilinear triangle bounded by three concave arcs of circle with opening angles $\pi /3$ and radius  $r$, pairwise mutually tangent at a common endpoint; 
\smallskip
\item[--]
$\wt _{r}$  the region bounded by two concave arcs of circle with opening angles $\pi /2$ and radius  $r$,  mutually tangent at a common endpoint, and a line segments tangent to such
arcs at their (noncommon) endpoints;
\item[--]  $\wwt _{r}$ the region bounded by a concave arc of circle with opening angle $2 \pi /3$ and radius  $r$, and two  line segments tangent to such
arc at its endpoints, forming an angle of $\pi/3$. 
\end{itemize}
\end{definition}

\begin{figure}[ht]
\begin{center}
\vskip -1.5cm
\includegraphics[scale=0.22]{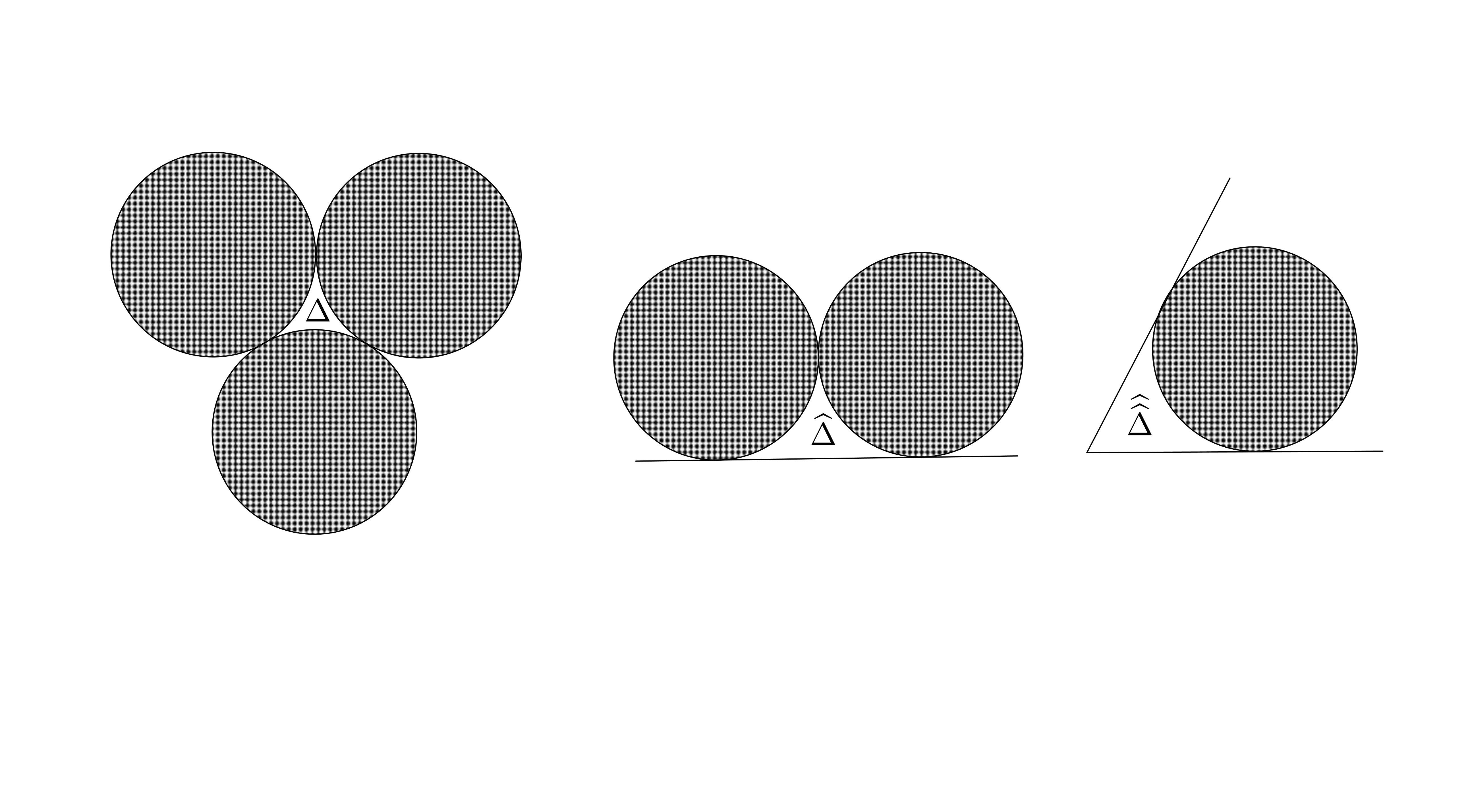}
\vskip -1.5 cm
\caption{The three types of empty regions: $\t, \wt, \wwt$}
\label{fig:odd2}
\end{center}
\end{figure}

\vskip -1.5cm

\begin{proposition}[area of the empty chamber]\label{p:void} 
For a fixed $p\geq 1$, let $\{\Omega _1, \dots, \Omega _k\}$ be an optimal cluster for problem \eqref{f:problemp} 
which satisfies the properties stated in Proposition \ref{p:optimal}. 
 Then the area of the empty chamber 
$\Omega _0 = \mathcal T \setminus \bigcup _{j=1}^k\overline  \Omega _j$
satisfies 
\begin{equation}\label{eq:inn}
|\Omega _0| \geq (2k-2) |\t _{r_*}| +   3    |\wwt _{r_*}|
\end{equation}
\end{proposition} 

 \proof

We call {\it empty room} a collection of (one or more) connected components of the empty chamber which are enclosed by  
a face of the canonical graph, and which are {\it not} of type $\wwt$ (namely are not around a corner of $\mathcal T$).

We observe that the area of the empty chamber can be estimated from below by the  global area of all the empty rooms
(the inequality may be strict because there may be cells touching $\partial \mathcal T$ which are not connected by any outer edge to  
$\R ^2 \setminus \overline {\mathcal T}$).

Then, we proceed to minimize the global area  of the empty rooms. To that aim, we modify the canonical graph associated to the optimal cluster so that each face has exactly $3$ edges. 
The modification consists in adding a certain number of formal edges for every empty room having on its boundary more than $3$ arcs.

Given such an empty room $C_0$, 
there exists a family of $m \geq 3$ disks $D_1, \dots, D _m$ of centers $P _1, \dots, P _m$ and radii $r_1, \dots, r _m$, with    
$$\begin{array}{ll} 
& \displaystyle d (P _i, P _{i+1} ) = r _i + r _{i+1}  \qquad \ \forall i = 1, \dots, m-1
\\
\noalign{\medskip}
& \displaystyle d ( P _i , P _j)  \geq r _i + r _j \qquad \qquad \forall i, j \in \{ 1, \dots, m \}, \ { |i-j| \ge 2},
\end{array}
$$
(where $d ( \cdot, \cdot )$ denotes the Euclidean distance and  
all angles $\angle{P_{i-1}P_iP_{i+1}} $ are strictly less than $\pi$), such that  one of the following situation occurs:

\begin{itemize}
\item[(a)]  $d (P _1, P _m ) = r _1 + r _{m}$, and $\partial C_0 \subseteq \partial D _1 \cup \dots \partial D _m$; 

\smallskip 
\item[(b)] $d (P _1, P _m ) > r _1 + r _{m}$, $D_1$, $D_m$ are tangent to one side $S$ of $\mathcal T$, and $\partial C_0 \subseteq S\cup \partial D _1 \cup \dots \partial D _m$; 
; 

\smallskip 
\item[(c)]  $d (P _1, P _m ) > r _1 + r _{m}$, $D_1$, $D_m$ are are tangent to two  consecutive   sides $S'$, $S''$ of $\mathcal T$, and $\partial C_0 \subseteq S' \cup S'' \cup \partial D _1 \cup \dots \partial D _m$.

\end{itemize}

 Now, according to the above cases (a)-(b)-(c), the modification of the graph runs as follows. 

If we are in situation (a), we label the cells around $C_0$ by $1, \dots, m$, and then we add edges joining the couples
$$(m, 2), (2, m-1), (m-1, 3),\dots $$
(see Figure \ref{fig:odd}).

\begin{figure} [ht]
\begin{center}
\includegraphics[scale=0.6]{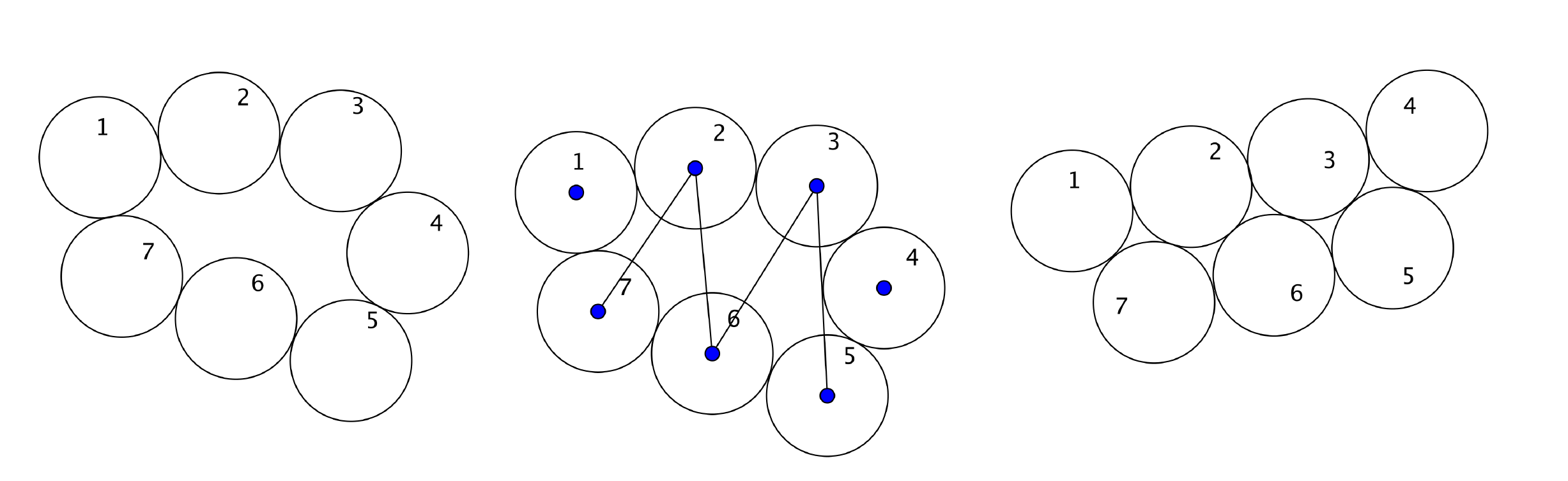}
\caption{Formal topological modification of the graph in situation (a)}
\label{fig:odd}
\end{center}
\end{figure}

If we are in situation (b) or (c),  we do the same kind of procedure starting with the edge $(1,m)$, namely we add edges joining the couples
$$(1, m), (m, 2), (2, m-1), (m-1, 3),\dots $$
(see Figure \ref{fig:line}).

\begin{figure}[ht]
\begin{center}
\includegraphics[scale=0.3]{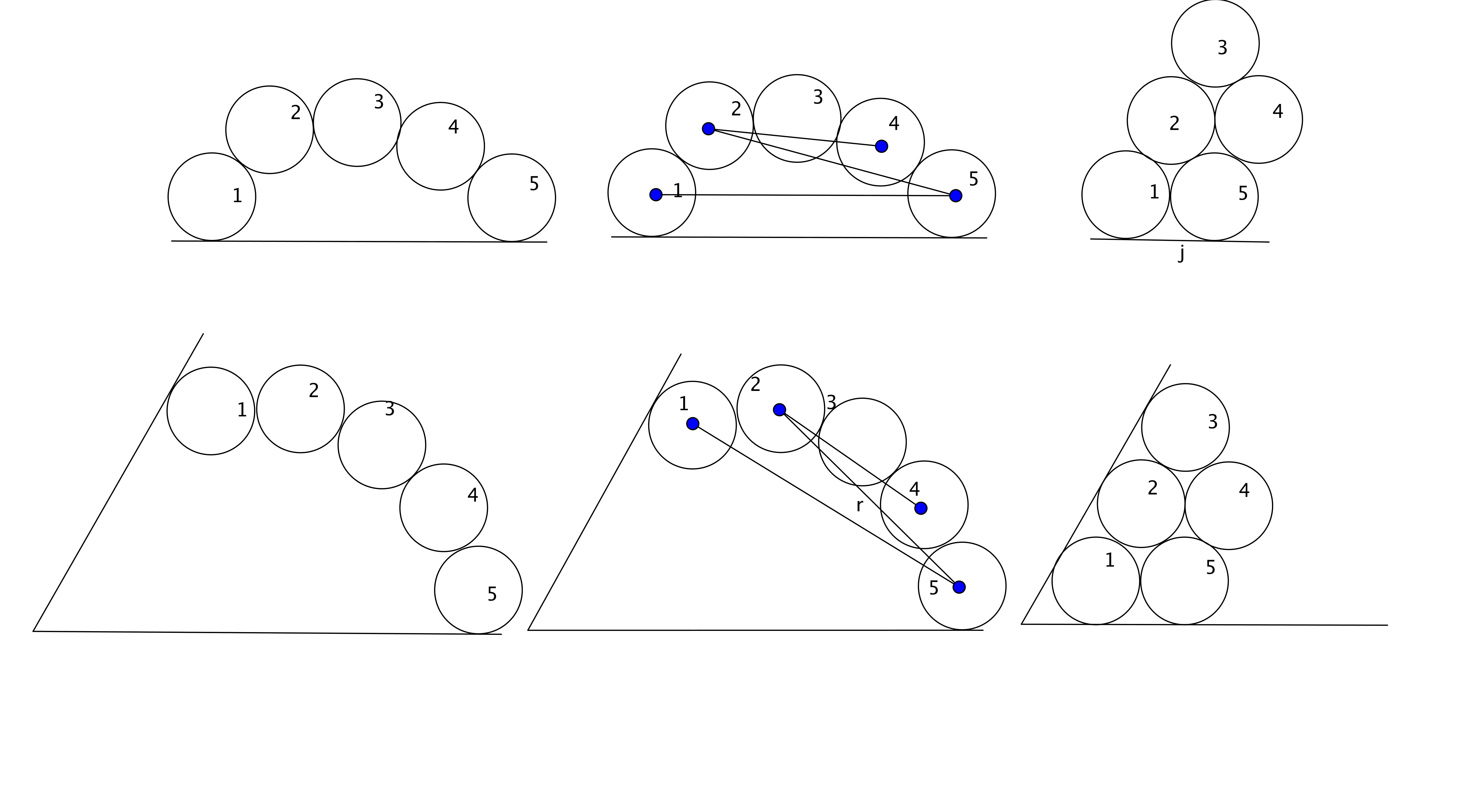}
\caption{Formal topological modification of the graph in situations (b) and (c)  }
\label{fig:line}

\end{center}
\end{figure}

{\it Notice that the fact that the radii of the disks in Figures \ref{fig:odd} and \ref{fig:line} are all equal 
is not relevant to the present topological purposes, and in any case can be a posteriori justified by the results in the Appendix.}

By construction, for the graph thus modified, each face has exactly $3$ edges. 
Hence, we have $2E=3F$. Recalling that $V = k +1$, the Euler formula $V-E+F=2$ gives $F=2k-2$. Then some easy but lengthy geometrical arguments, that we postpone to the Appendix (see Lemmas \ref{p:triangle1}, \ref{p:triangle2}, and \ref{p:triangle3}), imply that the  global area of all the empty rooms is not smaller than $(2k-2)  \t_{r^*}$, 
plus the contribution coming from   $3$  curvilinear triangles in the corners, each one of area $|\wwt_{r^*}|$.

\section{An intermediate result on the inner Cheeger boundary}\label{sec:bdry}

\begin{definition}\label{d:innerCheeger}
Let $\Omega$ belong to the class $\mathcal A$ introduced in Definition \ref{defA}. Let
$\Gamma:=\partial \Omega$, and let $r := h (\Omega) ^ {-1}$ be the radius of the free arcs.

We call {\it inner Cheeger boundary of $\Omega$} the  ``inner parallel curve at distance $r$ from $\Gamma$''
(namely the set of points in $\Omega$ lying at distance $r$ from $\Gamma$), 
endowed  with the same orientation as $\Gamma$.  
\end{definition}

\begin{remark}\label{rem:bf01}
We can make the following observations.

\begin{itemize}
\item [(i)]
The inner Cheeger boundary $\Gamma _r$ may have self intersection points.

\item[(ii)]  If  $\Gamma^l$, $l= 1, \dots, 2 \Lambda$, are the arcs of $\Gamma$ according to Definition \ref{defA}, 
we can decompose $\Gamma _r$ as $\Gamma ^ 1 _ r \cup \dots \cup \Gamma ^{2 \Lambda } _r$, where $\Gamma ^ l_r$ denotes the inner parallel curve at distance $r$ from $\Gamma ^l$, and the free arcs are labelled with an odd number. 
Then,  
 for $l$ odd the inner parallel curve  $\Gamma ^l _r$ is formally reduced to a point. For $l$ even, 
$\Gamma ^l _r$ is uniquely determined as follows:  if $\Gamma ^ {l}$ is an arc of circle with center $C ^l$ and nonzero curvature $K ^l$, $\Gamma ^l _r$ is the arc of circle obtained by applying an homothety of center $C^l$ and ratio $1-\frac{K(\Gamma^{l})}{h(\Omega)}$ to  $\Gamma^{l}$; if $\Gamma ^ {l}$ is a line segment, then $\Gamma ^ {l} _r$ is the line  segment obtained by moving $\Gamma ^ {l}$  in the direction of the inner normal to $\Gamma$ at distance $r$ from its original position.   
\end{itemize}
\end{remark}

\begin{definition}\label{d:innerarea}
Let $\Omega$ belong to the class $\mathcal A$ introduced in Definition \ref{defA}, and let $\Gamma _r$ denote its inner Cheeger boundary according to Definition \ref{d:innerCheeger}. 
We call {\it inner Cheeger area of $\Omega$} the oriented area enclosed by $\Gamma_r$, and we denote it by $A (\Gamma _r)$. Namely, 
$$A (\Gamma _r) = \sum _h m(U_h) |U_h| \, , $$ where  the sum is extended to the bounded connected components $U _h$ of $\R^2 \setminus \Gamma_r$,  $|U _h|$ is the Lebesgue measure of $U _h$, and the number $m(U_h) \in \Z$ is the index of any point of $U_h$ with respect to the oriented curve $\Gamma_r$. 
\end{definition}

 The following result is crucial to our purposes. It can be regarded as a transposition, valid within the class $\mathcal A$, 
of  a well-known result for the inner Cheeger set of convex bodies due to Kawohl and Lachand-Robert (see Theorem 1 in \cite{KaLR});
we also refer to \cite{lns} for a recent extension to domains `without necks'.

\begin{proposition}[representation via inner Cheeger set] \label{p:representation}
Let $\Omega$ belong to the class $\mathcal A$ introduced in Definition \ref{defA}. 
Let  $\Gamma _r $ and $A(\Gamma_r)$ be its inner  Cheeger boundary and inner Cheeger area according to Definitions \ref{d:innerCheeger} and \ref{d:innerarea}. There holds:
\begin{eqnarray}
& A  ( \Gamma_r ) = \pi r ^ 2 & \label{f:area}
\\ \noalign{\smallskip} 
& |\Omega| = r \mathcal H ^ 1 (\Gamma _r) + 2 \pi r ^ 2 \,. & \label{f:rep}
\end{eqnarray} 
\end{proposition}

\proof  
We are going to show the validity of the following Steiner-type formulas:
\begin{eqnarray}
& \mathcal H ^1 (\partial \Omega) = \mathcal H ^ 1 (\Gamma _r) + 2 \pi r  & \label{per}
\\ 
& |\Omega| = A  ( \Gamma _ r) + r \mathcal H ^ 1 (\Gamma _r) + \pi r ^ 2  & \label{area} 
\end{eqnarray} 
Taking into account that $\frac{\mathcal H ^ 1(\partial \Omega)}{ |\Omega|} = \frac{1}{r}$, the required equalities \eqref{f:area} and \eqref{f:rep} will follow. 

To prove \eqref{per}-\eqref{area}, we need to consider the following angles (see Figure \ref{fig:innerCheeger}):
$$\begin{array}{ll}
& \theta_i  \text{:= the opening angles of odd arc of radius $r$}; 
\\  \noalign{\medskip}
& \alpha _i   \text{:= the opening angles of even arcs of radius
$r_i>r$ and negative curvature}; 
\\ \noalign{\medskip}
& \beta _i \text{:= the opening angles of even arcs of radius
$r_i>r$ and positive curvature.}
\end{array}
$$

\begin{figure}
\begin{center}
\includegraphics[scale=0.5]{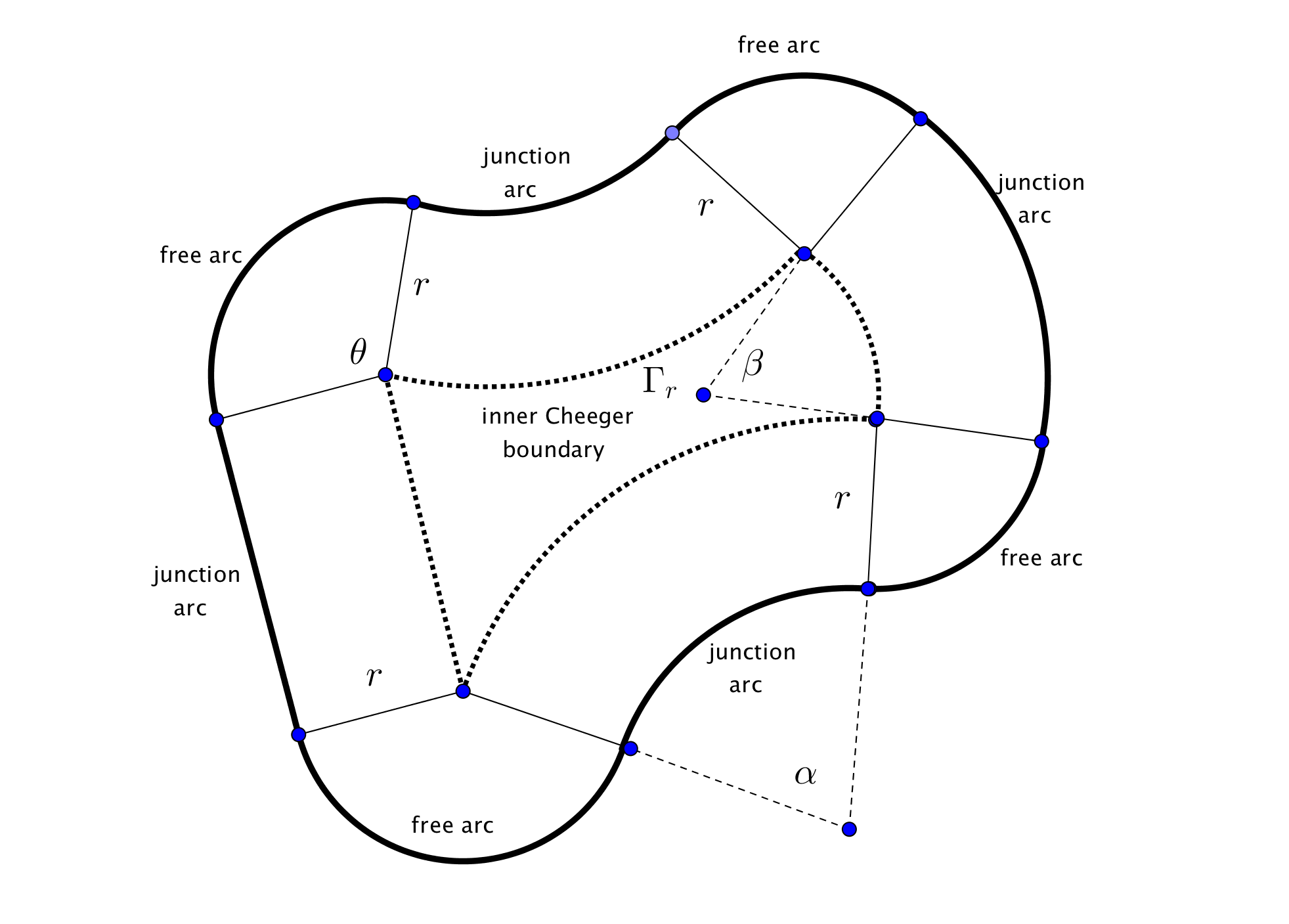}
\caption{The geometry of the inner Cheeger set}
\label{fig:innerCheeger}
\end{center}
\end{figure}

We claim that the above angles obeys  the following rule:
\begin{equation}\label{f:rule}
\sum _{i}   \theta _i  + 
\sum _{i} \beta _i - 
\sum _{i} \alpha _i  
= 2  \pi  \,,
\end{equation}
where the sums are extended to the families of all angles of each type.

In order to prove \eqref{f:rule}, we consider the oriented polygon $P$ having vertices
$$O _1 T _{1,2} O _2 T _{2, 3} \dots O _N  T _{\Lambda, 1} \, , $$
where $T _{i, i +1}$ is the touching point between the curves $\Gamma ^ i$ and $\Gamma ^ {i +1}$, $\Lambda$ is as in Remark \ref{rem:bf01}, and $O _i$ is defined as follows:
$$\begin{array}{ll}
&   \text{-- if $\Gamma ^i$ is a circular arc, $O_i$ is the center of the disk containing $\Gamma ^i$}; 
\\  \noalign{\medskip}
&   \text{-- if $\Gamma ^i$ is a line segment, $O_i$ is an arbitrary point of $\Gamma ^i$.}
\end{array}
$$
We are going to compute the sum of the inner angles of the polygon $P$. To that aim, 
we distinguish the following types of pairs of consecutive curves in $\partial \Omega$:
$$\begin{array}{ll}
&   \text{Type 1: an odd arc of radius  $r$ - an even arc of radius $>r$ and negative curvature}; 
\\  \noalign{\medskip}
&   \text{Type 2: an odd arc of radius  $r$ - an even arc of radius $>r$ and positive curvature}; 
\\  \noalign{\medskip}
&   \text{Type 3: an odd arc of radius  $r$ - an even line segment.}
\end{array}
$$
We denote by $N _i$ the number of pairs of type $i$ which are contained in $\partial \Omega$. 
For $i = 1, 2, 3$, each pair of type $i$ contributes with $4$ vertices of $P$. 

Therefore, 
the sum of the inner angles of $P$ equals 
\begin{equation}\label{sum1}
[4 (N _1 + N _2 + N _3 )
-2]\pi\,.
\end{equation} 
On the other hand, the contribution given to the sum of the inner angles by the pairs of each type is listed below. 
Setting for brevity $\widehat T _{j, j +1}:= \angle  O _j T _{j, j +1} O _{j+1}$, and $\widehat O _j := T_{j-1, j} O _j T_{j, j+1} $, we have
$$\begin{array}{lllll}
&   \text{Type 1}: \qquad \widehat O _{2i-1} =  2 \pi - \theta, \quad &\widehat T _{2i-1, 2i} = \pi,  \quad &\widehat O _{2i} = \alpha , \quad &\widehat T_{2i, 2i+1} =  \pi; 
\\  \noalign{\medskip}
&   \text{Type 2}: \qquad \widehat O _{2i-1} =  2 \pi - \theta, \quad &\widehat T _{2i-1, 2i} = 0, \quad &\widehat O _{2i} =  2 \pi - \beta, \quad &\widehat T_{2i, 2i+1} = 0  ; 
\\  \noalign{\medskip}
&   \text{Type 3}:\qquad \widehat O _{2i-1} =  2 \pi - \theta, \quad &\widehat T _{2i-1, 2i} = \frac{\pi}{2},  \quad &\widehat O _{2i} = \pi, \quad &\widehat T_{2i, 2i+1} =  \frac{\pi}{2}.
\end{array}
$$
From this table we see that the sum of the inner angles of $P$ is:
\begin{equation}\label{sum2}
-\sum _i  \theta _i + \sum _i \alpha _i - \sum _i \beta _i 
+ 4 \pi (N _1 + N _2 +  N _3) \,. 
\end{equation}
Imposing the equality between the expressions in \eqref{sum1} and \eqref{sum2}, we obtain  the required formula \eqref{f:rule}.

Now, relying on the equality \eqref{f:rule}, we are ready to prove \eqref{per}-\eqref{area}.  We introduce the following notation:
$$\begin{array}{ll}
&   \text{-- if $\Gamma ^i$ is a circular arc, we denote by $r (= h (\Omega))$ or $r _i ( >r)$ its radius}; 
\\  \noalign{\medskip}
&   \text{-- if $\Gamma ^i$ is a line segment, 
we denote by $\ell _i$ its length}.
\end{array}
$$

By direct computation, recalling Definition \ref{d:innerCheeger}, we have
\begin{equation} \label{areaC}  \mathcal H ^ 1 (\partial \Omega) = 
\sum _{i }  \theta _i r  + 
\sum _{i }  \alpha _i r _i +\sum _{i }  \beta _i r _i   + \sum _i \ell _i 
\end{equation}
\begin{equation} \label{areaCr}  
\mathcal H ^ 1 (\Gamma _r)  =
\sum _{i }  \alpha _i (r _i + r)  +
\sum _{i }  \beta _i (r _i - r) + \sum _i \ell _i 
  \,.
\end{equation} 
By subtracting and using \eqref{f:rule}, we get 
$$\mathcal H ^ 1(\partial \Omega) - \mathcal H ^ 1 (\Gamma _r) = \Big [  \sum _{i}  \theta _i -  \sum _{i}   \alpha _i +  \sum _{i}  \beta _i    \Big ] r = 2 \pi r \,,  $$
which proves \eqref{per}. 

Now we turn our attention to \eqref{area}. 
We observe that
\begin{equation}\label{calcularea}
|\Omega| = \int_\Gamma x \, dy \qquad\hbox{ and } \qquad A (\Gamma _r)  = \int_{\Gamma_r} x \, dy  \,. \footnote{
Indeed, by the Gauss-Green Theorem, if $U$ is a Jordan domain with positively oriented, piecewise smooth boundary $\Gamma _U$, for any  $f \in C ^ 1 ( \Omega)$  it holds
$\int \!\!  \int _{U} \frac{\partial f }{\partial x } \, dx \, dy = \int _{\Gamma_U} f \, dy$;
in particular, taking $f(x, y) = x$, we get
$|U| = \int _{\Gamma_U} x \, dy$. Applying this formula respectively to $\Omega$ and to the bounded connected components of $\R ^ 2 \setminus \Gamma _r$, we obtain the equalities in \eqref{calcularea}. }
\end{equation}

To compute the above integrals, we use the decompositions
$$\Gamma = \Gamma ^ 1 \cup \dots \cup  \Gamma ^{2 \Lambda }  \qquad \text{ and } \qquad \Gamma _r= \Gamma ^ 1_r \cup \dots \cup  \Gamma ^{2 \Lambda} _r  \, ,$$
and we introduce the oriented line segments
$$S^r_i:= [T^r _{i-1, i} , T _{i-1, i} ]\,, $$ 
where  $T _{i-1, i}$  is the touching point between $\Gamma ^{i-1}$ and $\Gamma ^i$, and 
 and $T^r _{i-1, i}$ is the touching point between $\Gamma ^{i-1}_r$ and $\Gamma ^i_r$ (with the conventions $\Gamma ^0:=  \Gamma ^{2 \Lambda}$ and $\Gamma _ r ^0: = \Gamma _ r ^{2 \Lambda} $). 
We have
  \begin{equation}\label{f:scomposta} \begin{array}{ll} \displaystyle \int _\Gamma x \, dy - \int _{\Gamma _r} x \, dy  & \displaystyle = \sum _{i =1} ^{2 \Lambda}   
\Big [  \int _{\Gamma^i}  x \, dy - \int _{\Gamma _r ^i } x \, dy  \Big ] \\ 
\noalign{\medskip}  
 & \displaystyle  =  \sum _{i =1} ^ {2 \Lambda}  
\Big [  \int _{\Gamma^i}  x \, dy - \int _{\Gamma _r ^i } x \, dy  + \int_ { S^r _i}    x \, dy - \int _ {S^r_{i+1} } x \, dy   \Big ] \,.
\end{array} 
\end{equation}
By construction, for every $i = 1, \dots, N$, the curve $S^r_i+ \Gamma ^ i - S^r_{i+1}   - \Gamma ^i_r$ is the positively oriented boundary of a Jordan domain $D_i$. Thus, each addendum of the last sum in \eqref{f:scomposta} is equal to the Lebesgue measure of $D_i$, which is easily computed as follows: 
\begin{itemize}
\item[--] if $\Gamma ^i$ is an arc of radius $r$ and opening angle $\theta _i$, then $\Gamma ^ i _r$ is a concentric arc of radius $0$, so that 
$$|D_i| = \frac{\theta _i}{2} r ^2\, ;$$ 
\item[--] if $\Gamma ^i$  is a negatively curved arc of radius $r _i >r$ and opening angle $\alpha _i$, then $\Gamma ^ i _r$ is a concentric arc of radius radius $r_i+r$, so that$$|D_i| = \frac{\alpha _i}{2} \big [ ( r _i + r  ) ^ 2 - r _i ^ 2 \big ] = \alpha _i r _i r + \frac{\alpha _i}{2} r ^ 2\,;$$
\item[--] if $\Gamma ^i$  is a positively curved arc of radius $r _i >r$ and opening angle $\beta _i$ , then  $\Gamma ^ i _r$ is a concentric arc of radius $r_i-r$, so that $$|D_i| = \frac{\beta _i}{2} \big [ r _i ^ 2 - ( r _i - r  ) ^ 2 \big ]  = \beta _i r _i r - \frac{\beta _i}{2} r ^ 2\,;$$
\item[--] if $\Gamma ^i$  is a line segment of length $\ell _i$, then $\Gamma ^ i _r$ is a parallel line segment of the same length, so that
$$|D_i| = \ell _i r\,.$$

 \end{itemize} 
Summing up, we obtain 
$$\begin{array}{ll} |\Omega| - A ( \Gamma _r) & \displaystyle = \int _\Gamma x \, dy - \int _{\Gamma _r} x \, dy   =  \sum _{i =1} ^{2 \Lambda}   
|D_i|    \\  \noalign{\medskip} & \displaystyle = 
 \Big [\sum _{i} \frac{\theta _i}{2} +  \sum _{i}  \frac{\alpha _i}{2} 
 - \sum _{i} \frac{\beta _i}{2}    \Big ] r ^ 2 + \Big [  \sum _{i}  {\alpha _i r _i}  + \sum _{i}  {\beta _i r _i}   \Big ] r  + \sum _i \ell _i r  \,.
\end{array}
$$
Next, we subtract from the above expression $r \mathcal H ^ 1 (\Gamma _r)$, that we compute from \eqref{areaCr}. 
We get
$$ |\Omega| - A ( \Gamma _r)  - r \mathcal H ^ 1 (\Gamma _r) =  \Big [\sum _{i} \frac{\theta _i}{2} -  \sum _{i}  \frac{\alpha _i}{2} 
 + \sum _{i} \frac{\beta _i}{2}    \Big ] r ^ 2    \,.$$
Eventually, we invoke \eqref{f:rule} and we obtain \eqref{area}. \qed

\bigskip

\section{Proof of Theorems \ref{t:truehoney}, \ref{t:truehoney_bis}, and \ref{t:truehoney_ter}.}\label{sec:proof} 

\subsection{Proof of Theorem \ref{t:truehoney}}

Let us prove inequality \eqref{aux}.
We take an optimal partition $\{\Omega _1, \dots, \Omega _k \}$  for problem \eqref{f:problemp}.
We set 
$$\begin{array}{ll}
& h (\Omega_j) = h _j = r _j ^ {-1} \qquad \forall  j = 1, \dots, k
\\ \noalign{\bigskip}
& \displaystyle \max _{j = 1 , \dots, k} h _j = h _* = r _* ^ {-1} \,.
\end{array}
$$ 

We now divide the proof of \eqref{aux} in 4  steps. 

\medskip
\underbar {\it Step 1.} For every $j = 1, \dots, k$, we apply Proposition \ref{p:representation} to the cell $\Omega _j$. We denote by  $\Gamma _{r _j}$ the inner Cheeger boundary of $\Omega _j$.

By \eqref{f:rep}, we have
\begin{equation}\label{radice}
|\Omega _j | = r_j \mathcal H ^ 1 (\Gamma _{ r_j}) + 2 \pi r_j ^ 2 \,.
\end{equation}

multiplying by $h_j ^ 2$, we have
\begin{equation}\label{moltiplicata}
h _j ^ 2 |\Omega _j | - h_j \mathcal H ^ 1 (\Gamma _{ r_j}) = 2 \pi  \,.
\end{equation}

We look at the polynomial $x \mapsto p_j (x):=  |\Omega _j| x ^ 2   -  \mathcal H ^ 1 (\Gamma _{ r_j})  x $. 
By \eqref{radice} we have that $h _j$ is larger that the largest root of $p$, namely 
\begin{equation}
h _j > \frac{\mathcal H ^ 1 (\Gamma _{ r_j}) }{|\Omega _j |} \,.
\end{equation} 
Then, since $h _ * \geq h _j$, we have $p _j ( h _*) \geq p _j ( h _j )$, and we infer from \eqref{moltiplicata} that
\begin{equation}\label{pre-endstep2}
h _* ^ 2 |\Omega _j | \geq h_* \mathcal H ^ 1 (\Gamma _{ r_j}) + 2 \pi  \,.
\end{equation} 
We conclude this step by summing the above inequality over $j$:
\begin{equation}\label{endstep2}
h _* ^ 2  \sum _{j=1} ^k |\Omega _j | \geq h_* \sum _{j=1} ^k  \mathcal H ^ 1 (\Gamma _{ r_j}) + 2 \pi k \,.
\end{equation}

\medskip
\underbar {\it Step 2.} In order to estimate from below the r.h.s.\ of \eqref{endstep2}, we are going to use Hales hexagonal isoperimetric inequality.  
According to \eqref{f:area} we  
have $A (\Gamma _{r_j}) = \pi r_j ^ 2$ for every $j$. Therefore
$$ A  \Big ( \frac{ \Gamma _{r_j} }{\sqrt{\pi r _* ^ 2 }} \Big ) = \frac {\pi r_j ^ 2}  {\pi r _* ^ 2 } \geq 1\, , $$
so that 
$$\min \Big \{  A  \Big ( \frac{ \Gamma _{r_j} }{\sqrt{\pi r _* ^ 2 }} \Big )   , 1 \Big \} 
= 1 \,.$$

On the inner Cheeger boundary  $\Gamma_{ r_j }$, we fix the following family $\mathcal N _j$ of nodes: first,  we take as nodes all 
the points which are at distance $r _j$ from an odd free arc of $\partial \Omega _j$ 
(in equivalent terms, any such node joins two arcs of $\Gamma _{r_j}$ which are parallel to two consecutive even junction arcs of $\partial \Omega _j$ separated by a free arc); 
then, we add the following ``exceptional nodes'': if a border junction arc of $\partial \Omega _j$ contains  
 different segments lying on $\partial \mathcal T$, we take as nodes also the points in $\Gamma _{r_j}$ which are at distance $r_j$ from the endpoints of all these segments.  
 
Accordingly, we write  $\Gamma _{r_j} = \Gamma ^ 1 _{r_j} \cup \dots \cup \Gamma ^{N_j} _{r_j}$, where $N_j$ is the cardinality of the family of nodes $\mathcal N _j$, and, for $i = 1, \dots, N _j$, $\Gamma ^i_{r_j}$ is the (oriented) portion of $\Gamma _{r_j}$ delimited by two consecutive nodes $n_{i-1},  n _{i}$ (with the convention $n _0 = n _{N_j}$).

Now, we set
$T(\Gamma_{ r_j })$ the (truncated) deficit associated to the  oriented curve $\Gamma_{ r_j }$ and the family $\mathcal N _j$.  
Namely,
$$ T(\Gamma_{ r_j }): = \sum _{i=1} ^ {N _j} x (\Gamma^i _{r_j}) \wedge 1 \vee ( -1) \,,$$
where $x  (\Gamma^i _{r_j})$ is the signed area enclosed by the oriented curve $\Gamma^i_{r_j} \cup [n _{i}, n _{i-1}]$.

 Then,
Hales' hexagonal isoperimetric inequality \cite[Theorem 4]{Hales} gives
\begin{equation}\label{hales}
\mathcal H ^ 1\Big  (\frac{ \Gamma _{ r_j }}{ \sqrt {\pi r_*^ 2}} \Big )  \geq - \frac{1}{\pi r_*^ 2} T (\Gamma _{ r_j } )\c- (N  _j - 6) 0.0505 + 2 \c .
\end{equation}
 where $2 \c$  is the perimeter of the unit area regular hexagon. 

We now multiply \eqref{hales} by $\sqrt \pi$, and we sum over $j = 1, \dots, k$, taking into account that:
\begin{eqnarray} -\sum _{j=1} ^ k T (\Gamma _{ r_j } ) \geq 0 & \label{deficit} 
\\-\sum _{j=1} ^ k (N _j - 6)  \geq 0 & \label{mean6}
\end{eqnarray}  
 To obtain \eqref{deficit},  we  observe that each 
piece of curve $\Gamma ^ i _{r_j}$ is an arc of circle, possibly of zero curvature  (in particular, thanks to the addition of the exceptional nodes,
$\Gamma ^ i _{r_j}$ cannot be a broken line). If $\Gamma ^ i _{r_j}$ has zero curvature, it produces a zero deficit. 
 If $\Gamma ^ i _{r_j}$ has a nonzero curvature, it is parallel to a junction arc between $\Omega_j$ and another cell $\Omega_l$, so that  it produces two deficits.  Assume the  curvature of $\Gamma ^ i _{r_j}$,  seen from $\Omega_j$,   has a negative sign. Then the deficit  $x(\Gamma ^i_{ r_j } )$ (the dashed region in Figure \ref{fig:excess}) has a negative sign, while the deficit  $x(\Gamma^i _{ r_l} )$ (the black region in Figure \ref{fig:excess}) has a positive sign, being in absolute value smaller than the previous one. This is simply due to the fact that the two regions which contribute to the deficit are homothetic, with a ratio larger than one. We conclude that \eqref{deficit} holds true. 

 To obtain \eqref{mean6}, we observe that
\begin{equation}\label{scompo}
\sum _{j=1} ^kN  _j \leq \sum _{j=1} ^k \Lambda _j  + 3  \leq 6k \, ,
\end{equation}
where the second inequality holds true by Proposition \ref{p:nb}.  To obtain the first equality in \eqref{scompo}, we observe that
 the arcs of $\Gamma _{r_j}$ are in bijection with the junction arcs in $\partial \Omega _j$, except for the extra arcs created in $\Gamma _{r_j}$ by exceptional nodes. 
Now, the maximal  possible number of such  extra arcs  is $3$:  in fact,
recalling that distinct segments contained into a unique border junction arc must lie on different sides of $\mathcal T$ (see end of Section \ref{sec:important}), we see that the configuration containing the highest number of extra arcs (equal $3$) is the one in which  there are $3$  border junction arcs, each one containing $2$ segments (lying on consecutive sides of $\mathcal T$).

\begin{figure}
\begin{center}
\includegraphics[scale=0.5]{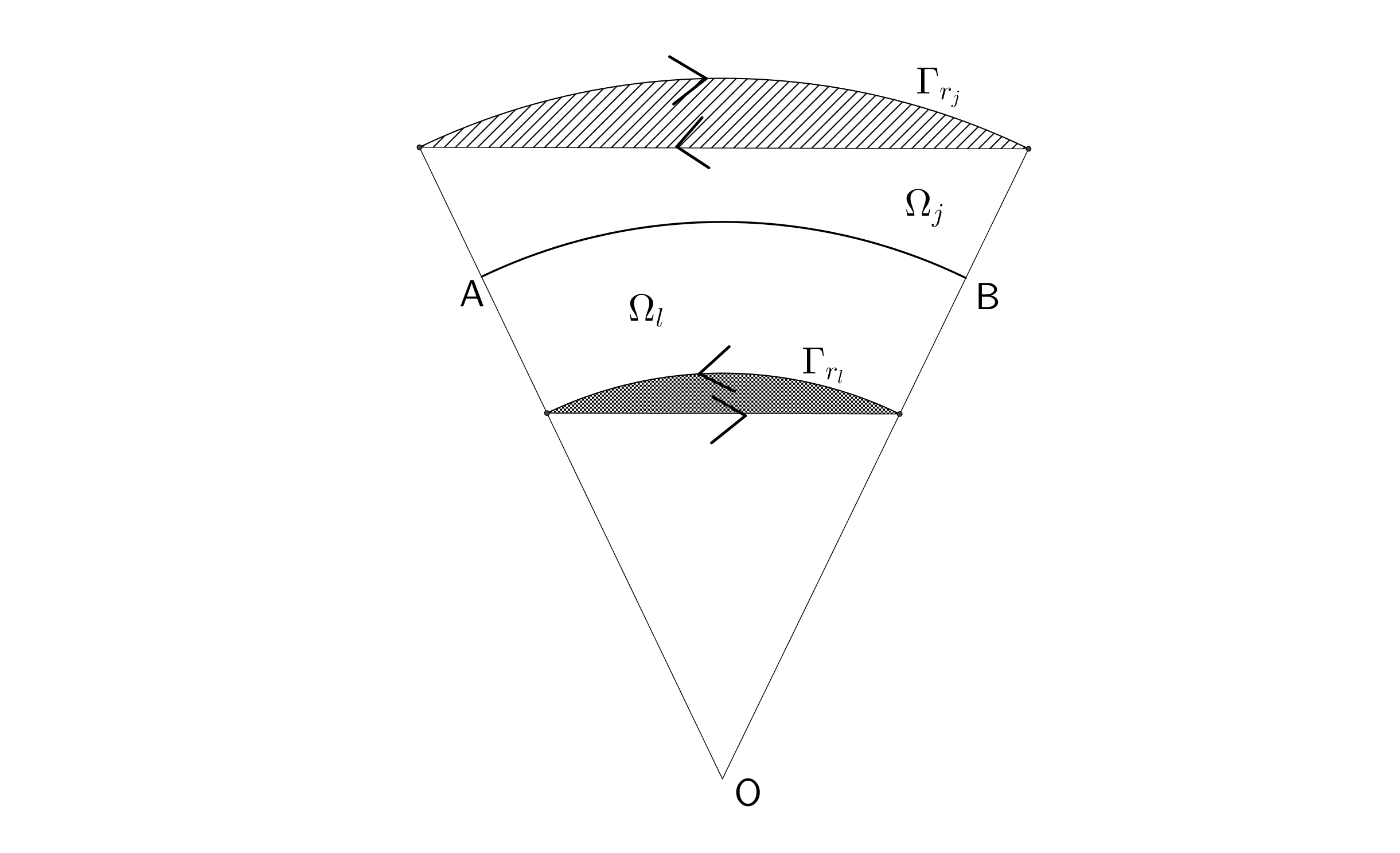}
\caption{The positive contribution of the polygonal deficit}
\label{fig:excess}
\end{center}
\end{figure}

\smallskip
By \eqref{hales}, \eqref{deficit}, and \eqref{mean6}, we get  the following lower bound for the r.h.s.\ of \eqref{endstep2} 
\begin{equation}\label{boundbelow}
h_*  \sum _{j=1} ^ k   \mathcal H ^ 1 (\Gamma _{r_j}) + 2 \pi k \geq k \big [  2\, \sqrt \pi  \c + 2 \pi\big ] \,.
\end{equation}

\medskip
\underbar{\it Step 3.} By elementary computation, we know that $|\t _{r_*}| = \frac{r_*^2}{2 } ( 2 \sqrt 3 - \pi)$  and $
|\wwt _{r_*}|=\frac{r_*^2}{3 } ( 3\sqrt 3 - \pi)$.
 In order to estimate from above the l.h.s.\ of  \eqref{endstep2}, we exploit Proposition \ref{p:void}. 
Inequality \eqref{eq:inn}, together with the computations above, imply 
$$|\Omega _0| \geq 2k |\t _{r_*}|.$$ 
So 
$$ |\Omega _0| \geq    k ( 2 \sqrt 3 - \pi )  \frac{1}{h _* ^ 2} \,, $$ 
yielding
\begin{equation}\label{boundabove} 
h_* ^ 2 \sum _{j=1}^k |\Omega _j |  \leq h_* ^ 2 ( |\mathcal T| - |\Omega _0| ) \leq  \big [ h _* ^ 2  |\mathcal T| -  k ( 2 \sqrt 3 - \pi) \big ]\,.
\end{equation}

\medskip
\underbar{\it Step 4 (conclusion).}  We put together the information coming from the previous three steps. By \eqref{endstep2}, \eqref{boundbelow}, and \eqref{boundabove}, we have: 

$$ \frac{|\mathcal T|} {k} h_* ^ 2 \geq \pi + 2 \sqrt 3 + 2\, \sqrt \pi  \c = \Big [ \frac{12\sin(\frac{\pi}{6})+\sqrt{12\pi \sin(\frac{\pi}{3})}}{\sqrt{12 \sin(\frac{\pi}{3})}} \Big ] ^ 2 = [h (H)] ^2\,. $$
This concludes the proof of \eqref{aux}.

\medskip We now turn to the proof  of \eqref{tesihoney}.  In order to deduce it from  \eqref{aux}, it is enough to check that 
$\widetilde M _{k, p} (\Omega)$ converges to $M _k (\Omega)$ in the limit as $p \to + \infty$.  
In fact, 
this follows from the inequalities
\begin{equation}\label{cfr} M_k ({\mathcal T}) \leq \widetilde M _{k, p} ({\mathcal T}) \leq k ^ {1/p} M _k ({\mathcal T}) \,, 
\end{equation}
which are readily obtained as follows. 
If $\{\Omega ^ p _1, \dots, \Omega ^ p _k \}$ is an optimal solution for $M _{k,p} ({\mathcal T})$, and 
$\{\Omega ^ \infty _1, \dots, \Omega ^ \infty _k \}$ is an optimal solution for $M _{k} ({\mathcal T})$, we have
$$
\begin{array}{ll}
& \displaystyle M _{k} ({\mathcal T}) \leq  \max_{j = 1, \dots , k}  h  (\Omega^p _j ) = \widetilde M _{k, p} ({\mathcal T})  \\ 
\noalign{\medskip} 
& \displaystyle k [M _{k} ({\mathcal T}) ]^p  \geq \sum _{i=1} ^k h ^p (\Omega^\infty _j )  \geq  \sum _{i=1} ^k h ^p (\Omega^p _j ) \geq  \max_{j = 1, \dots , k} h ^p (\Omega^p _j )
= [\widetilde M _{k, p} ({\mathcal T})]^p \,.
\end{array}$$
\qed

\subsection{Proof of Theorem \ref{t:truehoney_bis}}

 Let $k$ be fixed.
Clearly, from the definition of ${M_k (\mathcal T_k})$ and from the geometry of $\mathcal T _k$, precisely since $\mathcal T _k$ contains the $k$-cluster made by $k$ copies of $H$, it holds
$$\frac{{|\mathcal T_k|}^\frac 12}{k^\frac 12}M_k({\mathcal T_k}) \leq h(H).
$$
Assume by contradiction that the strict inequality holds, namely
\begin{equation}\label{tesihoney_contrary}
M_k({\mathcal T_k}) = ( 1- \delta ) h(H) \frac{k^\frac 12}{{|\mathcal T_k|}^\frac 12} \, , \qquad \text{ with } \delta \in (0, 1)\,.
\end{equation}
Let $\mathcal T$ be the equilateral triangle of fixed area, say equal $1$. 
For every $\eta \in \N$, $\mathcal T$ contains a family of mutually disjoints $k$-triangles $\{\mathcal T^i _k\}_{i = 1, \dots, \eta}$, 
having  the same area, infinitesimal as $\eta$ tends to $+\infty$,  and such that 
\begin{equation}\label{area.bis}
\lim _{\eta \to + \infty} \big | \mathcal T \setminus \bigcup _{ i=1 } ^ \eta \mathcal T^i _k \big  |   = 0\,.
\end{equation}
By Theorem \ref{t:truehoney}, we have
\begin{equation}\label{uno}
M _{\eta k} (\mathcal T) \geq h (H) (\eta k) ^ {1/2}\,.
\end{equation}
On the other hand, by using assumption \eqref{tesihoney_contrary} (applied to each of the $k$-triangles $\{\mathcal T^i _k\}$, for $i = 1, \dots, k$), we infer that there exists a $(\eta k)$-cluster of $\mathcal T$ whose  cells have a Cheeger constant not larger than $ ( 1- \delta ) h(H) \frac{k^{\frac {1}{2}}} {{|\mathcal T^i_k|}^\frac 12}$. Thus, 
\begin{equation}\label{due}
M _{\eta k} (\mathcal T) \leq   ( 1- \delta ) h(H) \frac{k^{\frac {1}{2}}} {{|\mathcal T^i_k|}^\frac 12}\,.
\end{equation}
By combining \eqref{uno} and \eqref{due}, we obtain 
$$( 1 - \delta) \geq (\eta |\mathcal T ^i_k| )^ {1/2} \,.$$
In the limit as $\eta \to + \infty$, the above inequality gives a contradiction: indeed, in view of \eqref{area.bis}, we have $\lim\limits _{\eta \to +\infty} \eta |\mathcal T ^i_k| = |\mathcal T| = 1$.
\qed

\subsection{ Proof of Theorem \ref{t:truehoney_ter}  } 

 Once proved Theorem \ref{t:truehoney_bis}, the way Theorem \ref{t:truehoney_ter} is deduced is the same as in case of convex cells treated in \cite{bfvv17}. Thus we limit ourselves to indicate the strategy, referring to \cite{bfvv17} for the detailed arguments.  

First, one shows that the equality \eqref{tesihoney_bis}
 in Theorem \ref{t:truehoney_bis} extends to the case in which the $k$-triangle $\mathcal T _k$ is replaced by a ``$k$-cell'' $\Sigma _k$,  meant as a connected set of arbitrary shape obtained as the union of $k$ hexagons lying in a tiling of $\R ^2$ made by a family of copies of a regular hexagon. The passage from a $k$-triangle to a $k$-cell  is performed as follows. For simplicity, and without loss of generality, we can assume that $|\Sigma _k| = k$.  Since $\Sigma _k$ contains  a $k$-clusters made by $k$ copies of $H$, it holds $M _k (\Sigma _k ) \leq h (H)$. Assume by contradiction that $M _k (\Sigma _k ) < h (H)$. This means that there exists a $k$-cluster $\{\Omega _j\}$ of $\Sigma _k$ such that $\max _{j= 1, \dots, k} h (\Omega _j) < h (H)$. We can assume (up to shrinking a little bit the sets $\Omega _j$) that each of them is at positive distance from $\partial \Sigma _k$. Then we embedd $\Sigma _k$ into a big $k'$-triangle $\mathcal T _{k'}$, with $k' >k$, and we consider the $k'$-cluster of $\mathcal T _{k'}$ which is made by $ \Omega _j$ (for $j = 1, \dots, k$) union $ \widetilde H _j$ (for $j = 1, \dots, k'-k$), where $\widetilde H _j$ are slight deformations of the copies of $H$ contained into $\mathcal T _{k'}\setminus \Sigma _k$, constructed so that $h(\widetilde H _j) < h (H)$
(this can be done by continuity and since we have assumed ${\rm dist} (\Omega _j, \partial \Sigma _k)>0$). 
We have thus constructed a $k'$-cluster of $\mathcal T _{k'}$ in which each cell has a Cheeger constant strictly less than $h (H)$, against the equality \eqref{tesihoney_bis}.  

Now, using the equality \eqref{tesihoney_bis}
 for $k$-cells, it is possible to show separately the inequalities
$$\limsup_{k \to + \infty}\frac{|\Omega| ^ {1 /2}} {k ^ {1/2}}  M _{k}  (\Omega)  \leq
h ( H) \qquad \text{ and } \qquad
\liminf_{k \to + \infty}\frac{|\Omega| ^ {1 /2}} {k ^ {1/2}}  M _{k}  (\Omega)  \geq
h ( H)\, 
$$
via a blow up argument. More precisely, 
the upper bound inequality is proved by dilating $\Omega$ so that it is well approximated from {\it inside} with a $k$-cell, and using just the homogeneity and decreasing  monotonicity of $M_k (\cdot) $ by domain inclusion. The lower bound inequality is proved by dilating $\Omega$ so that it is well approximated frou {\it outside} with a $k$-cell, and using now, 
besides the behaviour of $M _k (\cdot)$ under dilations and inclusions,  
the crucial information that \eqref{tesihoney_bis}
 holds  for $k$-cells. 
\qed

\section {Appendix: geometrical estimates for the empty chamber}\label{sec:app} 

We give here three geometrical lemmas, in which we estimate  from below the area of  the region $V$ bounded by a ``closed chain" of consecutive tangent disks (Lemma \ref{p:triangle1}), 
by an ``open chain" of consecutive tangent disks and a segment (Lemma \ref{p:triangle2}),  and by an ``open chain'' of consecutive tangent disks and two line segments forming an angle of $\pi/3$ (Lemma \ref{p:triangle3}). 

These results are needed in the proof of Proposition \ref{p:void} in order to estimate from below the global area of all the empty rooms. 
More precisely, referring to the proof of Proposition \ref{p:void}, 
Lemma \ref{p:triangle1}, \ref{p:triangle2} and \ref{p:triangle3} concern respectively the area of an empty room of type (a), (b), and (c): it turns out to be
not smaller than the number of faces associated with the room in the modified graph times the area of a curvilinear triangle $\t_{r_*}$, with the addition of an extra curvilinear triangle $\wwt_{r_*}$ in case (c).

As usual, we denote by $d (\cdot, \cdot)$ the Euclidean distance. 

\begin{lemma}\label{p:triangle1}
Let $D_1, \dots, D _m$ be a family of $m \geq 3$ disks of centers $P _1, \dots, P _m$ and radii $r_1, \dots, r _m$ such that  
$$\begin{array}{lll} 
& \displaystyle d (P _i, P _{i+1} ) = r _i + r _{i+1} & \forall i = 1, \dots, m\quad  
\\
\noalign{\medskip}
& \displaystyle d ( P _i , P _j)  > r _i + r _j   & \forall i, j \in \{ 1, \dots, m \}, \ { |i-j| \ge 2}
\\
\noalign{\medskip}
& \displaystyle \angle{P_{i-1}P_iP_{i+1}}  < \pi  & \forall i= 1, \dots, m.
\end{array}
$$
(with the conventions $m+1=1$ and $0 = m$). 
 
Setting $V$  the complement in $\R ^2$ of the unbounded connected component of $\R ^2 \setminus \cup _{i=1} ^ m D _i $, and $r _*:= \min \{r_1, \dots , r _m \}$, there holds
$$|V| \geq (m-2) |\t_{r_*}|\, .$$
\end{lemma}

\proof
We search for a configuration of the disks $D _1, \dots , D _m$ which minimizes  the area of $V$.
The existence of an optimal configuration  is immediate, since we deal with a finite-dimensional problem.  However, since the constraints are not closed,  possibly an optimal configuration is degenerated, meaning it may exhibit  some aligned triple of consecutive  centers ($\angle{P_{i-1}P_iP_{i+1}} = \pi$)
and/or  some  touching non-consecutive discs ($d ( P _i , P _j)  =r _i + r _j$  with $|i-j| \ge 2$). 

The statement will be obtained by induction on $m$.

\begin{figure}[ht]
\begin{center}
\includegraphics[scale=0.4]{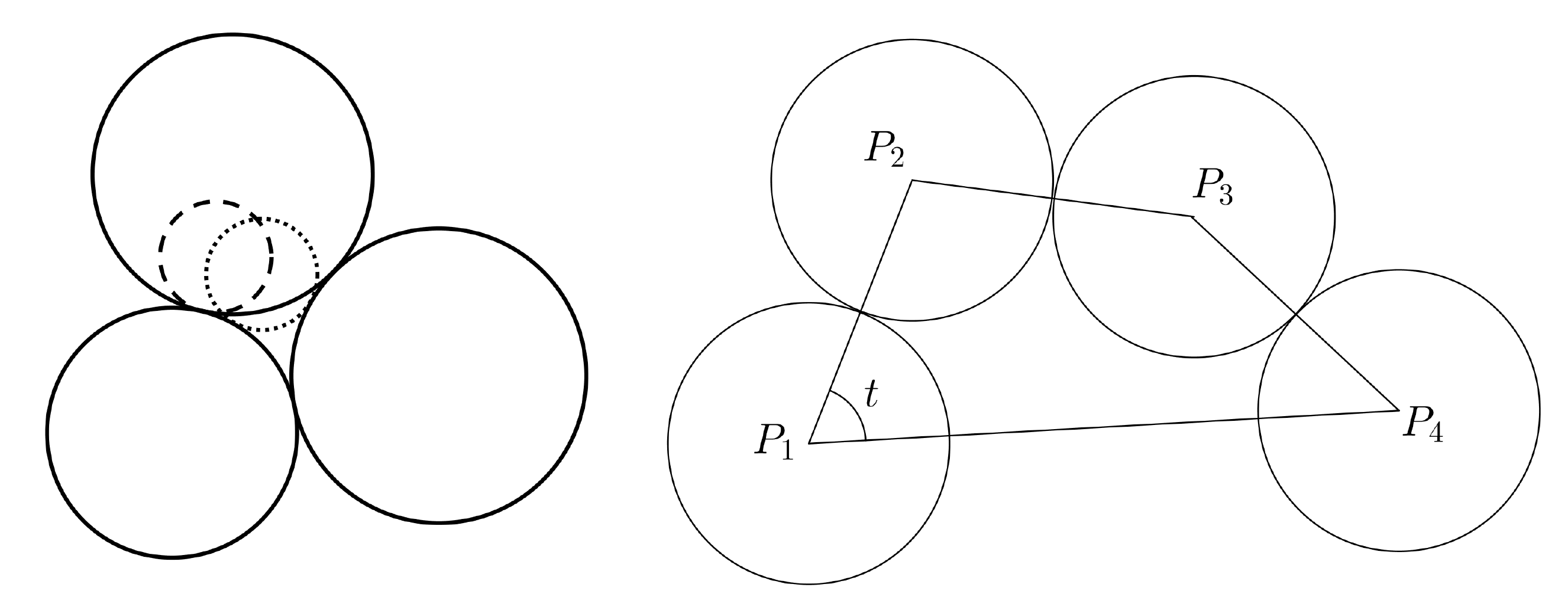}
\caption{Proof of Lemma \ref{p:triangle1}: initial step (left) and induction step (right)}
\label{fig:a}
\end{center}
\end{figure}

\underbar{Initial step.}  
Let $m = 3$. We have to show that the area of a curvilinear triangle bounded by three concave arcs of circle of radii $r _1, r _2, r _3$
is minimal when the three radii are equal. Let us show that, if one of the three radii, say $r_2$, is strictly larger than $r_*$, we can perturb the configuration of the three disks $\{D_1, D_2, D_3\}$  so to decrease the measure of the bounded connected component $V$ of $\R ^ 2 \setminus \big ( D_1 \cup D _ 2  \cup D_3  \big )$. The perturbation we consider is the following one: we keep $D_1$ and $D_3$ fixed, and we change $D_2$ into a new disk  $\widehat D _2$ which has radius $\widehat r _2$ strictly smaller than $r _2$ and is tangent to $D_1$ and $D_3$. 
Denoting by $\widehat V  $ the 
the bounded connected component  of $\R ^ 2 \setminus \big ( D_1 \cup \widehat D_2  \cup D _ 3  \big )$, we claim that the inclusion $\widehat V \subset V$ holds. 
Indeed, if we choose a system of coordinates so that $P_1 = (0, 0)$, and $P_3 = ( r_1 + r _3, 0)$,  we have  $P _2 = (x_0, y _0)$, with 
\begin{equation}\label{P2} \begin{array}{ll}
& \displaystyle x_0 = \frac{r_1^2+(r_1+r_3)^2+2 r_1
   r_2-r_3^2-2
   r_3 r_2}{2
   (r_1+r_3)}
\\ \noalign{\bigskip} & \displaystyle y _0 =  \sqrt{(r_1+r_2)^2-\frac{\left((r_1-
r_3) (r_1+r_3+2
   r_2)+(r_1+r_3)^2\right)^2}{4
   (r_1+r_3)^2}}\,.
   \end{array} 
   \end{equation}
   The geometry is represented in Figure \ref{fig:a}, left.

The derivatives of the angles 
$$\theta _1:= \angle P _2 P _1 P_3 = \arctan\Big ( \frac{y_0}{x_0} \Big ) \quad \text { and } \quad \theta _3:= \angle P _2 P _3 P_1 =  \arctan\Big ( \frac{y_0}{r_1+r_3-x_0} \Big ) $$ 
with respect to $r _2$ are positive, since they are easily computed as
\begin{equation}\label{derivatives}\begin{array}{ll}
& \displaystyle \frac{\partial \theta _1}{\partial r _2} =  \frac{2 r_1
   r_3}{(r_1+
   r_3) (r_1+r_2)
   \sqrt{-\frac{r_1
   r_3
   \left((r_1+r_3)^2
   -(r_1+r_3+2
   r_2)^2\right)}{(r_1+r_3)^2}}}
\\  \noalign{\bigskip}
& \displaystyle \frac{\partial \theta _3}{\partial r _2} = 
\frac{2 r _1 r_3}{
(r_1+r_3)
   (r_2+r_3)
   \sqrt{-\frac{r_1
   r_3
   \left((r_1+r_3)^2
   -(r_1+r_3+2
   r_2)^2\right)}{(r_1+r_3)^2}}}\,.  
\end{array}
\end{equation}
The inequalities $\frac{\partial \theta _1}{ \partial r _2}>0$ and  $\frac{\partial \theta _3}{ \partial r _2}>0$  imply the inclusion 
$\widehat V  \subset V$. In fact,  the following simple geometric argument shows that  $\partial D _2 \cap \partial \widehat  D _2 \cap \partial \widehat V  = \emptyset$. Let a disk of radius  $\widehat r _2$ roll from the position when it is externally tangent to $D _1$ at its tangency point with $D_2$, to the final position when it agrees with $\widehat D_2$. During this movement the intersection points between the boundary of the rolling disk and $\partial D _2$ are: $1$ point at the initial time, then $2$ points,  and eventually $2$, $1$ or $0$ points at the final time, all these intersections lying outside $\widehat V$.  In any case, at the final time no intersection point can belong to $\partial \widehat V$, {\it i.e.} $\partial D _2 \cap \partial \widehat  D _2 \cap \partial \widehat V  = \emptyset$.

\bigskip
\underbar{Induction step.} Assume the statement  holds true for up to $m-1$ disks, and let us show it holds true also for $m$ disks. 
Two cases may occur for an optimal configuration of $m$ disks. 

\smallskip {\it Case 1} : $d (P _i, P _j) = r _ i + r _j$ for some $i, j$ with $j \neq i +1$ (equivalently, $V$ is disconnected). With no loss of generality, let $i = 1$ and $ 2<j < m$. Consider the two disjoint families of disks
$\mathcal F':= \{ D _1, \dots, D _j \}$ and $\mathcal F '' := \{D_j,  D _{j+1} , \dots D _m , D_1\}$. They  have cardinalities $j$ and $m+2-j$, both strictly smaller than $m$. Hence,  letting $V '$ and $V ''$ be respectively the complements of the unbounded connected components of 
$\R ^ 2 \setminus \cup _{D_i \in \mathcal F'}  D _i$ and $\R ^ 2 \setminus \cup _{D_i \in \mathcal F''}  D _i$, by induction it holds
$$|V  ' |  \geq (j-2) |\t_{r_*} | \qquad \text { and } \qquad |V  '' |  \geq (m+2-j-2) |\t_{r_*}| =(m-j) |\t_{r_*}|\,.$$ 
Since by construction $V ' \cap  V '' = \emptyset$, and 
$V = V' \cup  V ''$, we obtain that $|V| \geq (m-2) |\t_{r_*}|$. 
This concludes the proof in Case 1.   

\smallskip {\it Case 2} : $d (P _i, P _j) > r _ i + r _j$ for all $i, j$ with $j \neq i +1$ (equivalently, $V$ is connected).
 We start by proving the following claim:   
\begin{equation}\label{equalradii}
r_i= r _*  \qquad \forall i = 1, \dots, m\,.
\end{equation} 
Namely, let us we show that, if one of the radii $r_1, \dots, r_m$ is strictly larger than $r_*$, we can perturb the configuration of the disks $\{D_1, \dots, D_m\}$  so to decrease the measure of  $V$. The perturbation we use is similar as the one considered in the initial step:  assuming without loss of generality that $r _2< r _*$, we keep all the disks fixed except $D _2$, and we change $D_2$ into a new disk  $\widehat D _2$ which has radius $\widehat r _2$ strictly smaller than $r _2$ and is tangent to $D_{1}$ and $D_3$. We remark that such a disk $\widehat D_2$ exists because  by assumption the centers $P_1$, $P_2$, and $P_3$ of the three involved disks are not aligned. Notice also that the perturbation we are considering is admissible because, in the case 2 we are dealing with, it holds $d (P_2, P _j) > r _ 2 + r _j$ for all $j \neq 1, 3$, which in particular ensures that the new configuration still satisfies the assumptions of the proposition. 
Denoting by $\widehat V $  the bounded connected component  of $\R ^ 2 \setminus \big ( D_1 \cup \widehat D_2\cup D _ 3    \big )$,  we claim that the inclusion $\widehat V \subset V$ holds.  The proof is similar as in the initial step. 
We choose a system of coordinates so that $P_1 = (0, 0)$, and $P_3 = (l , 0)$,  with $l >r_1+ r_3$. Accordingly, equations \eqref{P2} and \eqref{derivatives} are now replaced by  
\begin{equation}\label{P2.bis} \begin{array}{ll}
& \displaystyle x_0 = 
\frac{l^2+r_1^2+2
   r_1
   r_2- r_3^2-2
   r_3 r_2}{2 l}
\\ \noalign{\bigskip} & \displaystyle y _0 = \sqrt{(r_1+ r_2)^2-\frac{\left(l^2+(r_1-r_3) (r_1+ r_3+2
   r_2)\right)^2}{4 l^2}}
\,.
   \end{array} 
   \end{equation}
   and
   \begin{equation}\label{derivatives2}\begin{array}{ll}
& \displaystyle \frac{\partial \theta _1}{\partial r _2} = \frac{(l+r_1-r_3)
   (l-r_1+ r_3)}{l
   (r_1+r_2)
   \sqrt{-\frac{(l+r_1-r_3)
   (l-r_1+ r_3)
   \left(l^2-(r_1+ r_3
   +2
   r_2)^2\right)}{l^2}}} \\  \noalign{\bigskip}
& \displaystyle \frac{\partial \theta _3}{\partial r _2} =\frac{(l+r_1-r_3)
   (l-r_1+ r_3)}{l
   (r_3+r_2)
   \sqrt{-\frac{(l+r_1-r_3)
   (l-r_1+r_3)
   \left(l^2-(r_1+r_3+2
   r_2)^2\right)}{l^2}}}\,.
\end{array}
\end{equation}
Since the above derivative are positive, the  inclusion $\widehat V \subset V$ can be obtained as in the initial step, and the proof of \eqref{equalradii} is concluded. 

 To achieve our proof in case 2, it remains to show that a contradiction is reached as soon as we have $m \geq 4$.  Since we have proved condition \eqref{equalradii}, 
we are reduced to show the following assertion: given a number $m\geq 4$ of disks $D_1, \dots, D_m$  with equal radius $r_*$, and  centers $P_1, \dots , P _m$ such that $d (P_i, P_{i+1}) = 2 r_*$, $d (P _i, P _j) > 2r _ *$ if $j \neq i +1$, it is possible to perturb their configuration so to decrease the area of $V$. The perturbation we consider consists in keeping $D_1$ and $D_4$ fixed, and moving just $D_2$ and $D_3$, so that they remain tangent to each other and to $D_1$, $D_4$ respectively. Notice that such perturbation is  admissible because $d (P_2, P_j) > 2 r _*$ for all $j \neq 1, 3$ and similarly $d (P_3, P_j) > 2 r _*$ for all $j \neq 2, 4$.  Since \eqref{equalradii} holds, showing that the area of $V$ decreases is equivalent to showing that the area of the quadrilateral with vertices $P_1, P_2, P_3, P_4$ decreases. Assume without loss of generality that $2 r _*= 1$, and let $l$ the distance between $P _1$ and $P _4$. We have $l \geq 1$, with equality if $m = 4$. 
We name $t$ the angle formed by the side of length $l$ and one of its adjacent sides, see Figure \ref{fig:a}, right.

By the assumption on the angles $\angle{P_{i-1}P_iP_{i+1}}$, our quadrilateral is convex, and its area is given by
$$\varphi (t):= \frac{1}{4} \sqrt{\left(l^2-2 l
   \cos t+1\right)
   \left(-l^2+2 l \cos
   t+3\right)}+\frac{1}{2} l
   \sin t\,.$$
   
   We have
$$\varphi' (t) = \frac{1}{2} l \left(\frac{\sin
   t \left(-l^2+2 l \cos
   t+1\right)}{\sqrt{\left(l^
   2-2 l \cos t+1\right)
   \left(-l^2+2 l \cos
   t+3\right)}}+\cos
   t\right)\,.$$

Hence the inequality $\varphi' (t) \geq 0$  is equivalent to
$$\cos ^2t \left(l^2-2 l \cos
   t+1\right) \left(-l^2+2 l
   \cos t+3\right)-\sin ^2t
   \left(-l^2+2 l \cos
   t+1\right)^2 \geq 0\,, $$
and, in turn, to
$$\left(1-l^2\right) \left(4 \cos^2 t  - 4 l \cos t  + l^2 - 1\right) \geq 0\,. $$
   Taking into account that $l \geq 1$, we have   $\varphi' (t) \geq 0 $ if and only if 
   $$4 \cos^2 t  - 4 l \cos t  + l^2 - 1 \leq 0\,.$$ 
   The above inequality is satisfied on the interval $I _l:= \big [ 0, \arccos\big ( \frac{l-1}{2} \big) \big ] $. Indeed, setting $y := \cos t$, the  roots of the polynomial  $p _l (y):=4 y ^ 2 - 4l y + l ^ 2 -1 $ are $\frac{ l \pm 1}{2}$, so that $p _l (y) \leq 0$ on the interval $\big [ \frac{l-1}{2},  \frac{l+1}{2}\big ]$, which contains $I _l$   (because $1 \leq \frac{l+1}{2}$).

   Therefore,  the minimum of $\varphi (t)$ is achieved as $t \to 0$,  so that no nondegenerate quadrilateral can be optimal, 
   and our proof is achieved.  \qed

\begin{lemma}\label{p:triangle2}
Let $D_1, \dots, D _m$ be a family of $m \geq 3$ disks of centers $P _1, \dots, P _m$ and radii $r_1, \dots, r _m$, contained into a half-plane $H$ delimited by
a straight line tangent to $D_1$ and $D_m$,
such that  
$$\begin{array}{lll} 
& \displaystyle d (P _i, P _{i+1} ) = r _i + r _{i+1}  & \forall i = 1, \dots, m-1
\\
\noalign{\medskip}
& \displaystyle d ( P _i , P _j)  > r _i + r _j & \forall i, j \in \{ 1, \dots, m \}, \ { |i-j| \ge 2}
\\
\noalign{\medskip}
& \angle{P_{i-1}P_iP_{i+1}} < \pi & \forall i = 1, \dots, m 
\end{array}
$$
(where $P_0, P_{m+1}$ are the orthogonal projections on $\R ^ 2 \setminus H$ of $P_1, P_m$).

Setting  
$V$  the complement in $H$ of the unbounded connected component of $H \setminus \cup _{i=1} ^ m D _i $, and $r _*:= \min \{r_1, \dots , r _m \}$,  
there holds
$$|V| \geq (m-2) |\t_{r_*}| + |\wt _{r_*}|   (\geq (m-1) |\t_{r_*}| )  \, ;$$
\end{lemma}

\bigskip

\begin{figure}[ht]
\begin{center}
\includegraphics[scale=0.7]{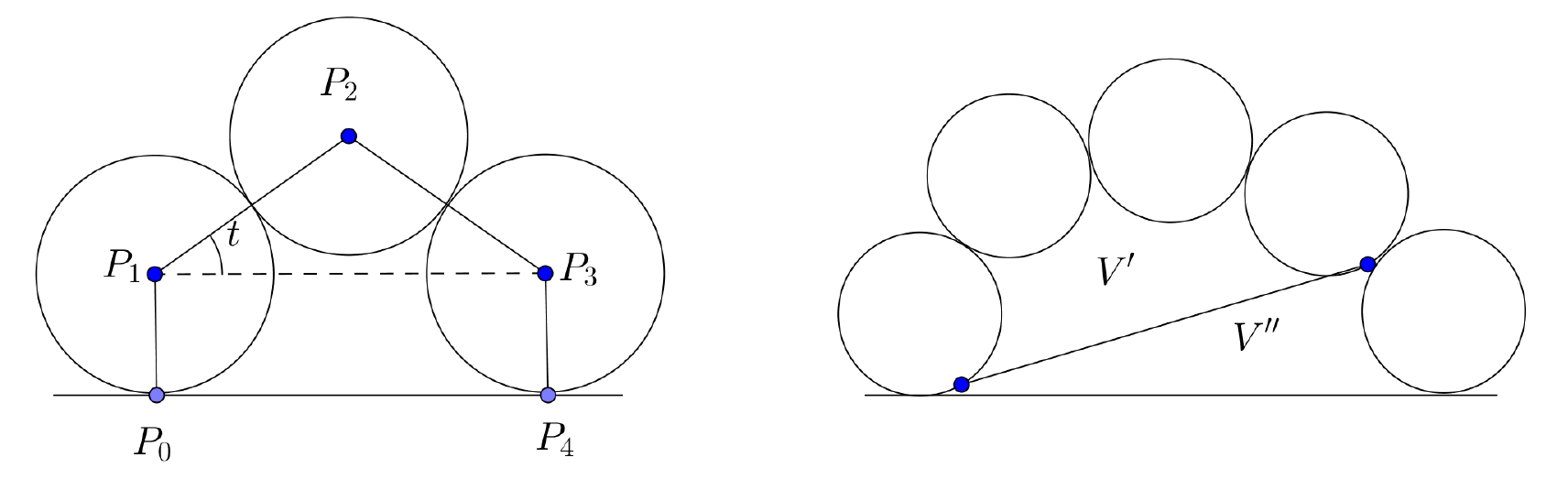}
\caption{Proof of Lemma \ref{p:triangle2}: initial step (left) and induction step (right)}
\label{fig:b}
\end{center}
\end{figure}

\proof 
Similarly as in the proof of Lemma \ref{p:triangle2}, we search for a (possibly degenerated) optimal configuration of the disks $D_1, \dots, D _m$ which minimizes the area of $V$, and we argue by induction on $m$.  

\smallskip
\underbar{Initial step.}  
Let $m = 3$. Let us show that the area of the region $V$  bounded by three concave arcs lying on disks $D_1, D_2, D_3$ with $(D_1, D_2)$ and $(D_2, D_3)$ mutually tangent, and a straight line $\gamma$ tangent to both $D_1$ and $D_3$, is minimal when the three radii are equal to $r_*$, and $V$ is the (disjoint) union $\t _{r_*} \cup \wt _{r_*}$. The fact that the three radii must be equal can be proved in the very same way as done in the initial step of the proof of Lemma \ref{p:triangle1}. 
Then we are reduced to minimize the area of the pentagon $P_0 P_1 P_2 P_3 P_4$ represented in Figure \ref{fig:b}, left.
Assuming without loss of generality that $2 r_* = 1$, and setting  $t:= \angle P_3 P_1 P_2$,  the area of such pentagon is given by
$$\varphi (t) := \cos t \, (1 + \sin t ) \,, \qquad t \in \big [0, \frac{\pi}{3} \big ]\,.$$ 
Then it is immediate to see that $\varphi' (t) \geq 0$ if and only if $t \in \big [0, \frac{\pi}{6} \big ] $, so that the minimum of $\varphi$ on the interval 
$\big [0, \frac{\pi}{3} \big ]$ is equal to
$$\min \big \{ \varphi (0), \varphi \big (\frac{\pi}{3} \big ) \big \} = \varphi \big (\frac{\pi}{3} \big )   = |\t _{r_*}| +  |\wt _{r_*}|\,.$$

\smallskip
\underbar{Induction step.}  
Assume the statement  holds true for up to $m-1$ disks, and let us show it holds true also for $m$ disks. Two cases may occur for an optimal configuration of $m$ disks.

\smallskip 
{\it Case 1}: $V$ is disconnected. Two subcases may occur:

\smallskip
{\it Case 1a}: The family $\{D_1, \dots D_k\}$ can be decomposed as the union of 
two disjoint subfamilies $\mathcal F'$ and $\mathcal F ''$, of cardinalities $j$ and $m+1-j$ (both strictly smaller than $m$), both satisfying the assumptions of Lemma \ref{p:triangle2}. In this case,   letting $V '$ and $V ''$ be respectively the complements  in $H$ of the unbounded connected components of 
$H \setminus \cup _{D_i \in \mathcal F'}  D _i$ and $H\setminus \cup _{D_i \in \mathcal F''}  D _i$, by induction it holds
$$|V  ' |  \geq (j-2) |\t_{r_*} | + |\wt_{r_*} |  \quad \text { and } \quad |V  '' |  \geq (m-j-1) |\t_{r_*}|+ |\wt_{r_*} |  \,.$$ 
Since by construction $V ' \cap  V '' = \emptyset$, and 
$V = V ' \cup  V ''$, we obtain 
 $$|V| \geq (m-3) |\t_{r_*}| + 2  |\wt_{r_*} | \geq (m-2) |\t_{r_*}|+ |\wt_{r_*} |  \,.$$ 

\smallskip
{\it Case 1b}: The family $\{D_1, \dots D_k\}$ can be decomposed as the union of 
two disjoint subfamilies $\mathcal F'$ and $\mathcal F ''$, of cardinalities $j$ and $m+2-j$ (both strictly smaller than $m$), such that one on them, say $\mathcal F'$, satisfies the assumptions of Lemma \ref{p:triangle1}, and the other one satisfies the assumptions of Lemma \ref{p:triangle2}. In this case,   letting $V '$ and $V ''$ be respectively the complements  in $H$ of the unbounded connected components of 
$H \setminus \cup _{D_i \in \mathcal F'}  D _i$ and $H\setminus \cup _{D_i \in \mathcal F''}  D _i$, by Lemma \ref{p:triangle2} and induction, it holds
$$|V  ' |  \geq (j-2) |\t_{r_*} |   \quad \text { and } \quad |V  '' |  \geq (m-j) |\t_{r_*}|+ |\wt_{r_*} |  \,.$$ 
Since by construction $V ' \cap  V '' = \emptyset$, and 
$V = V ' \cup  V ''$, we obtain 
 $$|V| \geq  (m-2) |\t_{r_*}|+ |\wt_{r_*} |  \,,$$ 

The proof of the induction step in Case 1 is concluded. 

\smallskip {\it Case 2}:   $V$ is connected. In this case, we  preliminary observe that the equality \eqref{equalradii} must be satisfied, otherwise the configuration cannot be optimal 
(the proof is exactly the same as in Case 2 of the induction step in the proof of Lemma \ref{p:triangle1}). 
Then, we 
consider a straight line tangent to both $D_1$ and $D _{m-1}$ such that $D_1, \dots,  D_{m-1}$ are contained into a half-plane $\widetilde H$ delimited by such line. We set  $V ' := V  \cap \widetilde H$ and $V '' := V  \cap (\R ^2 \setminus \widetilde H)$, see Figure \ref{fig:b}, right. By induction, we have $|V '| \geq (m-3) |\t _{r_*}|+|\wt_{r_*}|$. On the other hand, we observe that $V ''$ contains a copy of $\t_{r_*}$ (with strict inclusion, since we are dealing with case 2); hence we have $|V ''|  \geq |\t _{r_*}| $.   Since by construction $V ' \cap  V '' = \emptyset$, and 
$V = V ' \cup  V ''$, we obtain that $|V| \geq (m-2) |\t_{r_*}| +  |\wt_{r_*}|$, 
concluding the proof of the induction step also in Case 2.   

\begin{lemma}\label{p:triangle3}
Let $D_1, \dots, D _m$ be a family of $m \geq 3$ disks of centers $P _1, \dots, P _m$ and radii $r_1, \dots, r _m$, 
contained into a sector $S$ of opening angle $\pi/3$ delimited by two half lines tangent respectively to $D_1$ and $D_m$,  such that
$$\begin{array}{lll} 
& \displaystyle d (P _i, P _{i+1} ) = r _i + r _{i+1}  & \forall i = 1, \dots, m-1
\\
\noalign{\medskip}
& \displaystyle d ( P _i , P _j)  > r _i + r _j &  \forall i, j \in \{ 1, \dots, m \}, \ |i-j| \geq 2
\\
\noalign{\medskip}
& \angle{P_{i-1}P_iP_{i+1}}  < \pi & \forall i = 1, \dots, m
\end{array}
$$
(where $P_0, P_{m+1}$ are the orthogonal projections on $\R ^ 2 \setminus Q$ of $P_1, P_m$).

Setting $V$  the complement in $S$ of the unbounded connected component of $S \setminus \cup _{i=1} ^ m D _i $, and $r _*:= \min \{r_1, \dots , r _m \}$,  
there holds
$$|V| \geq (m-2) |\t_{r_*}| + |\wt _{r_*}|  + |\wwt_{r_*}| \geq (m-1) |\t_{r_*}|+ |\wwt_{r_*}| \, .$$
\end{lemma}

\begin{figure}[ht]
\begin{center}
\includegraphics[scale=0.5]{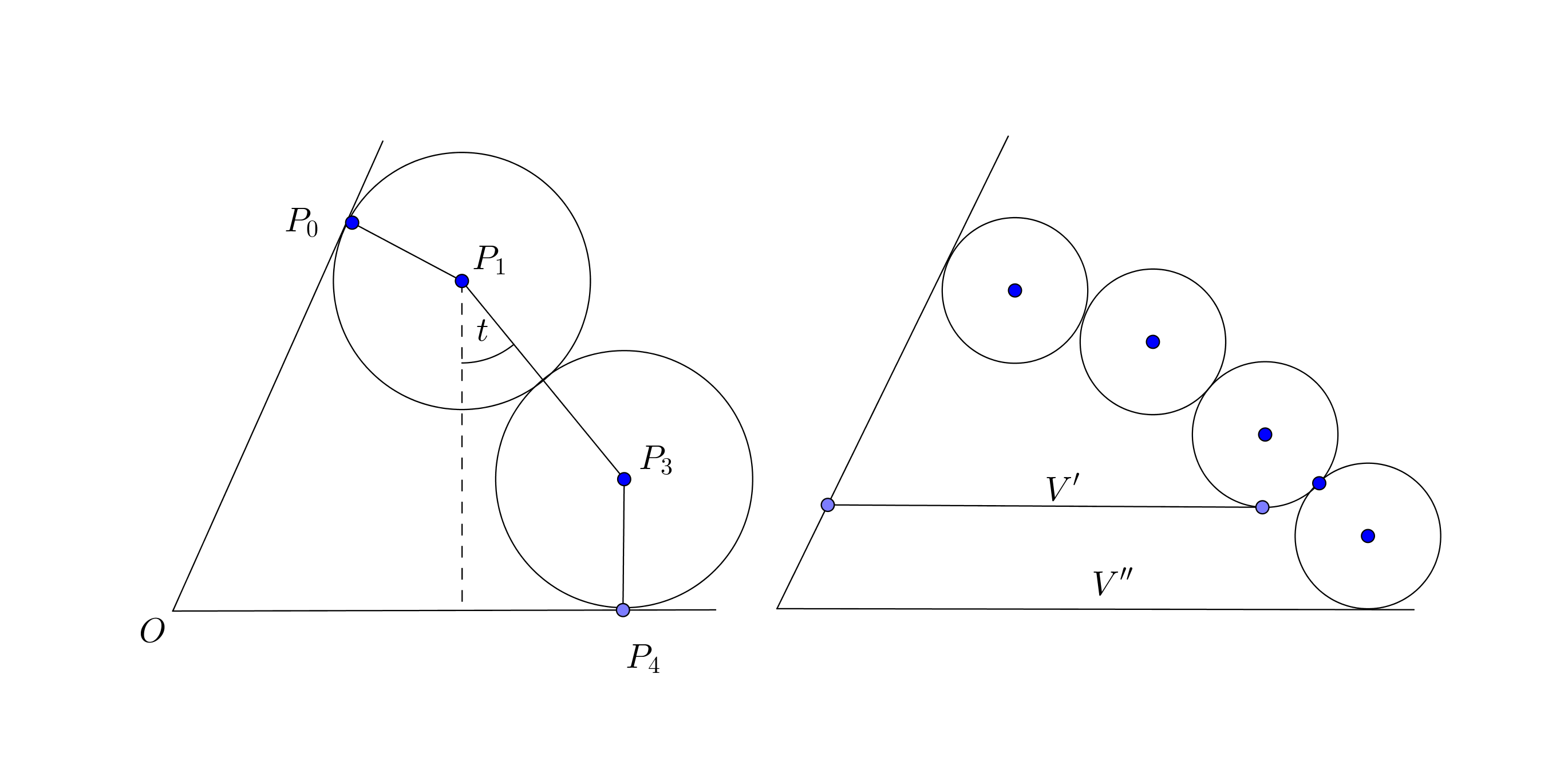}
\caption{Proof of Lemma \ref{p:triangle3}: initial step (left) and induction step (right)}
\label{fig:c}
\end{center}
\end{figure}

\proof

We argue again by induction on $m$. 

\smallskip
\underbar{Initial step.}  
Let $m = 3$. Let us show that the area of the region $V$  bounded by three concave arcs lying on disks $D_1, D_2, D_3$ with $(D_1, D_2)$ and $(D_2, D_3)$ mutually tangent, and two  half-lines forming an angle of $\pi/3$ and tangent respectively to $D_1$ and $D_3$,  is minimal when the three radii are equal to $r_*$, and $V$ is the (disjoint) union $\t _{r_*} \cup \wt _{r_*} \cup \wwt _{r_*}$. The fact that the three radii must be equal can be proved in the usual way as Lemmas \ref{p:triangle1} and \ref{p:triangle2}. 
Then we observe that, since by assumption the angle $\angle P _1 P_2 P_3$ is strictly less than $\pi$, there exists a straight line $\gamma$ tangent to both $D_1$ and $D_3$ such that $D_1, D_2, D_3$ are contained into a halfplane delimited by $\gamma$.  We set $V '$ the region delimited by  our three concave arcs and $\gamma$, and $V '' := V \setminus V '$. By Lemma \ref{p:triangle2}, we have 
$|V'| \geq   |\t _{r_*}|+ |\wt_{r_*}|$, so that we are reduced to show that $|V '' |  \geq  | \wwt _{r_*} |$.
It is not restrictive to prove the latter inequality in the setting when $d(P_1, P_3 ) = r _1 + r _3$ (because in such setting 
$ | V ' |=    |\t _{r_*}|+ |\wt_{r_*}|$,  and $|V ''|$ becomes strictly smaller than in the case when $d(P_1, P_3 ) > r _1 + r _3$).  
 It is readily seen that minimizing $|V''|$ is equivalent to minimizing the area of the pentagon $P_0 P_1 P_3 P_4 O$, being
 $O$ the origin of the two half-lines which delimit $S$, see Figure \ref{fig:c}, left.  
Setting $t:= \angle P_4 P_3 P_1 - \frac{ 2 \pi}{3}$, by elementary computations the area of such pentagon is given by
$$\varphi (t) :=  \frac{r_*^2}{3} \big [ 6 \sin t+3 \sin (2 t)+6
   \sqrt{3} \cos t+\sqrt{3}
   \cos (2 t)+4 \sqrt{3} \big ]  \,, \qquad t \in \big [0, \frac{\pi}{2} \big ]\,, $$ 
and the minimum  of the map $\varphi$ on the interval $ [0, \frac{\pi}{2} \big ]$ is attained at $t = \frac{\pi}{2}$.  This yields $\varphi(t) \geq (2+\sqrt 3) r_* ^ 2$, 
which corresponds to the case $V '' = \wwt _{r_*}$.

\smallskip
\underbar{Induction step.}  
Assume the statement  holds true for up to $m-1$ disks, and let us show it holds true also for $m$ disks. Two cases may occur for an optimal configuration of $m$ disks.

\smallskip 
{\it Case 1}:  $V$ is disconnected. 
Two subcases may occur: 

\smallskip
{\it Case 1a}: The family $\{D_1, \dots D_k\}$ can be decomposed as the union of 
two disjoint subfamilies $\mathcal F'$ and $\mathcal F ''$, of cardinalities $j$ and $m+1-j$ (both strictly smaller than $m$), such that one of them, say $\mathcal F'$, satisfies the assumptions of Lemma \ref{p:triangle2}, and the other one, say $\mathcal F''$ satisfies the assumptions of Lemma \ref{p:triangle3}. 
Letting $V '$ and $V ''$ be respectively the complements  in  $S$ of the unbounded connected components of 
$S \setminus \cup _{D_i \in \mathcal F'}  D _i$ and $S \setminus \cup _{D_i \in \mathcal F''}  D _i$, 
by Lemma \ref{p:triangle2} and induction, we have 
$$|V  ' |  \geq (j-2)|\t_{r_*} |+|\wt_{r_*}|  \qquad \text { and } \qquad |V  '' |  \geq (m-j -1) |\t_{r_*}|+ |\wt_{r_*}|+ |\wwt _{r_*}| \,.$$ 
Since by construction $V ' \cap  V '' = \emptyset$, and 
$V = V ' \cup  V ''$, we obtain
 $$|V| \geq (m-3) |\t_{r_*}| + 2 |\wt_{r_*}| +  |\wwt _{r_*}|
  \geq (m-2) |\t_{r_*}| +  |\wt_{r_*}| +  |\wwt _{r_*}| \, . $$ 

\smallskip
{\it Case 1b}: The family $\{D_1, \dots D_k\}$ can be decomposed as the union of 
two disjoint subfamilies $\mathcal F'$ and $\mathcal F ''$, of cardinalities $j$ and $m+2-j$ (both strictly smaller than $m$), such that one of them, say $\mathcal F'$, satisfies the assumptions of Lemma \ref{p:triangle1}, and the other one, say $\mathcal F''$ satisfies the assumptions of Lemma \ref{p:triangle3}. 
Letting $V '$ and $V ''$ be respectively the complements  in  $S$ of the unbounded connected components of 
$S \setminus \cup _{D_i \in \mathcal F'}  D _i$ and $S\setminus \cup _{D_i \in \mathcal F''}  D _i$,  
by Lemma \ref{p:triangle1} and induction, we have 
$$|V  ' |  \geq (j-2) |\t_{r_*} |   \qquad \text { and } \qquad |V  '' |  \geq (m-j ) |\t_{r_*}|+ |\wt_{r_*}|+ |\wwt _{r_*}| \,.$$ 
Since by construction $V ' \cap  V '' = \emptyset$, and 
$V = V ' \cup  V ''$, we obtain 
 $$|V|
  \geq (m-2) |\t_{r_*}| +  |\wt_{r_*}| +  |\wwt _{r_*}| \, . $$

The proof of the induction step in Case 1 is concluded. 

\smallskip 
{\it Case 2}:   no disk $D_j$, with $j \neq 1, m$, is tangent to a half-line which delimits $S$  (equivalently, 
$V$ is connected).
In this case, we  observe that the equality \eqref{equalradii} must be satisfied, with the usual proof. 
 Then, we 
consider a sector $\widetilde S$ of opening angle $\pi/3$ delimited by two half-lines: one of them is the same tangent to $D _1$ which delimits the original sector  $S$, and the other one is  tangent to $D_{m-1}$. We set  $V ' := V  \cap \widetilde S$ and $V '' := V  \cap (\R ^2 \setminus \widetilde S)$, see Figure \ref{fig:c}, right. By induction, we have $|V '| \geq (m-3) |\t _{r_*}| + |\wt _{r_*}|  + |\wwt _{r_*}|$. Moreover, $V ''$ contains a copy of $\t_{r_*}$ (with strict inclusion, since we are dealing with case 2); hence we have $|V '| \geq |\t _{r_*}| $.   Since by construction $V ' \cap  V '' = \emptyset$, and 
$V = V ' \cup  V ''$, we obtain that $|V|\geq (m-2) |\t _{r_*}| + |\t _{r_*}|  + |\wwt _{r_*}|$, 
concluding the proof of the induction step also in Case 2.

 \qed

\bigskip
\def\cprime{$'$}
\providecommand{\bysame}{\leavevmode\hbox to3em{\hrulefill}\thinspace}
\providecommand{\MR}{\relax\ifhmode\unskip\space\fi MR }
\providecommand{\MRhref}[2]{%
  \href{http://www.ams.org/mathscinet-getitem?mr=#1}{#2}
}
\providecommand{\href}[2]{#2}


\begin{thebibliography}{10}

\bibitem{ACMM01}
{L.} Ambrosio, {V.} Caselles, {S.} Masnou, and {J.M.} Morel, \emph{Connected
  components of sets of finite perimeter and applications to image processing},
  J. Eur. Math. Soc. (JEMS) \textbf{3} (2001), no.~1, 39--92.

\bibitem{bf17R}
{D.} Bucur and {I.} Fragal{\`a}, \emph{On the honeycomb conjecture for {R}obin
  {L}aplacian eigenvalues}, preprint CVGMT (2017).

\bibitem{bfvv17}
{D.} Bucur, {I.} Fragal\`a, {B.} Velichkov, and {G.} Verzini, \emph{On the
  honeycomb conjecture for a class of minimal convex partitions}, Arxiv
  Preprint, arXiv:1703.05383 (2017).

\bibitem{bhim96}
{K.} Burdzy, {R.} Holyst, {D.} Ingerman, and {P.} March, \emph{Configurational
  transition in a Fleming - Viot-type model and probabilistic interpretation of
  Laplacian eigenfunctions}, Journal of Physics A: Mathematical and General
  \textbf{29} (1996), no.~11, 2633.

\bibitem{CaffLin}
{L. A.} Caffarelli and {F. H.} Lin, \emph{An optimal partition problem for
  eigenvalues}, J. Sci. Comput. \textbf{31} (2007), no.~1-2, 5--18.

\bibitem{Car17}
M.~Caroccia, \emph{Cheeger {N}-clusters}, Calc. Var. Partial Differential
  Equations \textbf{56} (2017), no.~2, 56:30. \MR{3610172}

\bibitem{Ch}
{J.} Cheeger, \emph{A lower bound for the smallest eigenvalue of the
  {L}aplacian}, Problems in analysis ({P}apers dedicated to {S}alomon
  {B}ochner, 1969), Princeton Univ. Press, Princeton, N. J., 1970,
  pp.~195--199.

\bibitem{cbh05}
O.~Cybulski, V.~Babin, and R.~Holyst, \emph{Minimization of the {R}enyi entropy
  production in the space-partitioning process}, Phys. Rev. E (3) \textbf{71}
  (2005), no.~4, 046130, 10. \MR{2139992}

\bibitem{FT64}
L.~Fejes~T{\'o}th, Regular figures, A Pergamon Press Book, The Macmillan Co.,
  New York, 1964.

\bibitem{Hales}
{T. C.} Hales, \emph{The honeycomb conjecture}, Discrete Comput. Geom.
  \textbf{25} (2001), no.~1, 1--22.

\bibitem{HeHoTe}
{B.} Helffer, {T.} Hoffmann-Ostenhof, and {S.} Terracini, \emph{Nodal domains
  and spectral minimal partitions}, Ann. Inst. H. Poincar\'e Anal. Non
  Lin\'eaire \textbf{26} (2009), no.~1, 101--138.

\bibitem{Ivrii}
{V.} Ivrii, \emph{100 years of {W}eyl's law}, Bull. Math. Sci. \textbf{6}
  (2016), no.~3, 379--452.

\bibitem{KaLR}
{B.} Kawohl and {T.} Lachand-Robert, \emph{Characterization of {C}heeger sets
  for convex subsets of the plane}, Pacific J. Math. \textbf{225} (2006),
  no.~1, 103--118.

\bibitem{Trev}
{J. R.} Lee, {S. O.} Gharan, and {L.} Trevisan, \emph{Multiway spectral
  partitioning and higher-order {C}heeger inequalities}, J. ACM \textbf{61}
  (2014), no.~6, Art. 37, 30.

\bibitem{Leo}
{G. P.} Leonardi, \emph{An overview on the Cheeger problem}, Pratelli, {A.},
  Leugering, {G.} (eds.) New trends in shape optimization., International
  Series of Numerical Mathematics, Springer (Switzerland), vol. 166, 2016,
  pp.~117--139.

\bibitem{lns}
{G. P.} Leonardi, {R.} Neumayer, and {G.} Saracco, \emph{The Cheeger constant
  of a Jordan domain without necks}, Arxiv Preprint, arXiv:1704.07253 (2017).

\bibitem{Pa}
{E.} Parini, \emph{An introduction to the {C}heeger problem}, Surv. Math. Appl.
  \textbf{6} (2011), 9--21.

\bibitem{PaBo}
{E.} Parini and {V.} Bobkov, \emph{On the higher Cheeger problem}, Arxiv
  Preprint, arXiv:1706.07282 (2017).

\end{thebibliography}
\end{document}